\documentclass[12pt,reqno]{amsart}
    \makeatletter
    \renewcommand{\thepart}{\Roman{part}}
    \renewcommand\part{
    \clearpage \vspace*{-0.5cm} 
    \if@noskipsec \leavevmode \fi
    \@afterindentfalse
    \secdef\@part\@spart}
    \def\@part[#1]#2{%
    \ifnum \c@secnumdepth >\m@ne
      \refstepcounter{part}%
      \addcontentsline{toc}{part}{\thepart\hspace{5pt}#1}%
    \else
     \addcontentsline{toc}{part}{#1}%
    \fi
    {\parindent \z@ \centering
    \interlinepenalty \@M 
    \ifnum \c@secnumdepth >\m@ne
     \large\scshape \partname\nobreakspace\thepart:  #2 \par
    \else
     \Large\scshape \underline{#2}\par
    \fi
    } \nobreak \vskip 0.5cm \@afterheading}
    \makeatother
\usepackage{amsmath, amsthm, amscd, amsfonts, amssymb, bbm, graphicx, color, mathabx, tikz, caption, enumitem, mathtools, comment, pgfplots}
    \pgfplotsset{compat=1.16}
    \usepgfplotslibrary{fillbetween}
    \makeatletter
    \def\l@subsection{\@tocline{2}{0pt}{2.9pc}{5pc}{}}
    \def\l@subsubsection{\@tocline{2}{0pt}{5pc}{7.5pc}{}}
    \makeatother
\usepackage[margin=2.8cm]{geometry}
\usepackage[alphabetic]{amsrefs}

\usepackage[bookmarksnumbered, colorlinks, plainpages]{hyperref}
    \hypersetup{colorlinks=true, linkcolor=teal, citecolor=purple, filecolor=cyan, urlcolor=cyan}
\allowdisplaybreaks
    \mathtoolsset{showonlyrefs,showmanualtags} 
    
\numberwithin{equation}{section}

\newtheorem{mainthm}{Theorem}
    
\newtheorem{lemma}[equation]{Lemma}
\newtheorem{prop}[equation]{Proposition}
\newtheorem{mainprop}[mainthm]{Proposition}
    
\newtheorem{crlr}[equation]{Corollary}

\theoremstyle{definition}

\theoremstyle{remark}
\newtheorem{rmk}[equation]{Remark}
\newtheorem{example}[equation]{Example}

\newcommand{\N}{\mathbb N}
\newcommand{\R}{\mathbb R}
\newcommand{\Z}{\mathbb Z}

\newcommand{\T}{\mathbb T}
\renewcommand{\b}{\mathbf }

\newcommand{\btheta}{\boldsymbol\theta}
\newcommand{\bxi}{\boldsymbol\xi}

\renewcommand{\hat}{\widehat}

\newcommand{\one}{\mathbbm{1}}

\DeclareMathOperator{\dist}{dist}
\DeclareMathOperator{\lcm}{lcm}
\begin{document}
\title{Sharp Thresholds for Distance Patterns in Random Subsets of $\Z^d$}

\author[Giannitsi]{Christina Giannitsi}
    \address{Department of Mathematics, Virginia Tech, Blacksburg, VA, USA}
    \email {cgiannitsi@vt.edu}
\author[Palsson]{Eyvindur Ari Palsson}
    \address{Department of Mathematics, Virginia Tech, Blacksburg, VA, USA}
    \email {palsson@vt.edu}
\begin{abstract}
Let $d \geq 5$, $0 < \gamma < d - 2$, and $\Omega_N$ be the binomial random subset of $Q_N = [-N,N]^d \cap \Z^d$ with retention probability $p_N = N^{-\gamma}$.
We prove that, with failure probability of optimal exponential order, every subset $B \subseteq \Omega_N$ of fixed positive relative density realizes, at each scale $p_N^{-2/(d - 2)} \lesssim \lambda \lesssim N^2$, a squared distance of the form $q^2 \lambda$, where $q$ belongs to a fixed finite set depending only on the dimension and the density. The lower scale $p_N^{-2/(d - 2)}$ is sharp. As a consequence,the squared-distance set $D^2(B)$ of $B$ has maximal order $N^2$ and contains affine copies of every fixed finite subset of $\Z$. 

The main new input is a finite multidilate supersaturation theorem for dense subsets of $Q_N$, which, together with boundedness estimates for the associated spherical distance graphs down to the sharp scale, allows us to apply Schacht's transference theorem.
\end{abstract}
\maketitle
\tableofcontents

\section{Introduction}

A classic result in Ramsey theory is van der Waerden's theorem \cite{Waerden27}, which states that for all $k,r\in\mathbb{N}$, there exists $n$ such that every $r$-coloring of $1,2,\ldots,n$ contains a monochromatic $k$-term arithmetic progression. This was extended by Roth \cite{Roth53} who showed the existence of three-term arithmetic progressions in any subset of the integers of positive upper density and then further pushed to existence of arithmetic progressions of arbitrary length in the celebrated work by Szemer\'{e}di \cite{Szemeredi75}.

In geometric Ramsey theory, the patterns one seeks come from geometric considerations. A classic question, posed by Sz\'{e}kely \cite{Szekely83}, asks whether all sufficiently large distances can always be realized between pairs of points from any given measurable subset of $\mathbb{R}^2$ of positive upper Banach density. This was answered in the affirmative by Furstenberg, Katznelson and Weiss \cite{FKW90} using ergodic theory. Alternative proofs were given by Bourgain \cite{Bourgain86} using harmonic analysis, Falconer and Marstrand \cite{FM86} using more direct geometric arguments and more recently by Quas \cite{Quas09} using combinatorial methods. This result has been extended in many ways, notably by Bourgain \cite{Bourgain86} to pinned distances and simplices in the appropriate dimension. This line of research has seen a lot of work, particularly in the last decade; see for instance results on distances \cites{LM20b,W2025}, rectangles, simplices and their products \cites{HLM2017,LM18,DK21,DK22,LM22,DS25}, arithmetic progressions using $l^p$ metrics \cite{CMP17} and general finite configurations, either approximately \cite{Ziegler06} or under stronger density hypotheses \cite{FKY22}.

Geometric Ramsey theory questions can be asked in any setting with interesting geometry and, given the motivation from arithmetic progressions, it is natural to ask whether some suitable analogue of Sz\'{e}kely's question holds true in $\Z^d$, $d \geq 2$. Magyar established such a result for distances in \cite{M2008} and subsequently extended it to simplices in \cite{M2009}, using deep results from number theory on quadratic forms \cites{Kitaoka86,Siegel44}. Later, Lyall and Magyar \cite{LM2020} gave a new proof of Magyar's distance result and extended their methods to pinned distances and arbitrary finite trees. For an earlier unpinned tree result, see Bulinski \cite{Bulinski18}. Particularly relevant for our purposes is a finite quantitative result underlying their argument. For a dense set $A \subseteq Q_N$ satisfying an appropriate uniformity condition at an appropriate lattice scale, they show that, throughout a substantial range of $\lambda$, the normalized number of pairs $\b x,\b y \in A$ satisfying $|\b x-\b y|^2=\lambda$ is equal, up to a small error, to the square of the density of $A$. 
For infinite sets of positive upper Banach density, their density-increment argument reduces to a related uniformity condition on residue classes. Returning from the resulting affine sublattice to the original lattice introduces the dilation in their general distance theorem. Recently, Lott, Magyar and Ponagandla \cite{LMP26} obtained substantially stronger quantitative density bounds for similar copies of arbitrary nondegenerate integral simplices in sufficiently high dimensions, improving the quantitative bounds from Magyar's earlier work.

A newer direction in Ramsey theory involves proving so-called sparse random analogues of various well-known theorems. In the context of the theorems of Roth and Szemer\'{e}di, a classic setup is to consider a probabilistic model, where within $1, 2, \ldots, N$ each number is retained with some probability $p_N$ and then ask whether, with high probability as $N \rightarrow \infty$, every subset of fixed positive relative density in the resulting random set contains an arithmetic progression of the prescribed length. For Roth's theorem a sparse random analogue was obtained by Kohayakawa, {\L}uczak and R\"{o}dl \cite{KLR96}, who obtained a sharp threshold in the probability $p_N$. Along the way to proving their classic result on the existence of $k$-term arithmetic progressions in the primes, Green and Tao \cite{GT08} developed a transference argument showing that suitably pseudorandom sets satisfy a relative version of Szemer\'{e}di's theorem. The resulting threshold for random sets was, however, far from what one would expect to be sharp. For the true random model, the question of $k$-term arithmetic progressions was resolved by Schacht \cite{Schacht16} and, independently, Conlon and Gowers \cite{CG16}. Both obtained the sharp threshold on the probability $p_N$. More recently, Friedgut, Kuperwasser, Samotij and Schacht \cite{FKSS26} established sharp-threshold results for a broad class of Ramsey properties in random induced hypergraphs.

To our knowledge, sparse random analogues of this type have not previously been established for prescribed squared-distance scales in $\Z^d$. The goal of this paper is to develop such an analogue for the distance results of Magyar \cite{M2008} and Lyall and Magyar \cite{LM2020}. The finite results of Lyall and Magyar and the sparse transference theory of Schacht \cite{Schacht16} provide the two natural starting points, but substantial new input is needed to bring these theories together. The compact Lyall-Magyar argument gives quantitative spherical counts under a strong uniformity hypothesis, while Schacht's transference theorem requires a robust supersaturation statement for every sufficiently dense subset of the ambient set, together with suitable boundedness estimates for the associated configuration graphs. Our main deterministic input is a finite multidilate supersaturation theorem that provides such a statement for arbitrary dense subsets of $Q_N$. The appearance of several dilations is forced by congruence obstructions: no single dilation can satisfy the required counting estimate uniformly over all dense sets and scales. We then show that the resulting spherical distance graphs satisfy the boundedness conditions needed for Schacht's theorem uniformly throughout the relevant range of scales. In this way, we place the arithmetic distance problem into a form compatible with sharp sparse transference, bringing together the dense geometric framework of Lyall and Magyar with Schacht's probabilistic theory. The lower scale $p_N^{-2/(d-2)}$ is natural: since $|S_\lambda| \simeq_d \lambda^{(d-2)/2}$, it is precisely the scale at which $p_N |S_\lambda|$ becomes of constant order, and we show that the conclusion can fail below this scale. The exponential order of the failure probability is also optimal: it is $\exp\big(-c p_N |Q_N|\big)$ up to constants in the exponent, matching the probability that the random set $\Omega_N$ itself is empty.

\subsection{Our Main Results}
With probability at least $1-o(1)$ for all sufficiently large $N$, the random set $\Omega_N$ has the following resilience property. Every subset $B\subseteq\Omega_N$ containing at least a fixed proportion $\delta$ of the points of $\Omega_N$ realizes distances at every admissible scale. More precisely, there is a fixed finite collection of dilation factors, depending only on $d$ and $\delta$, such that for each admissible integer $\lambda$, the squared-distance set of $B$ contains $q^2\lambda$ for at least one dilation factor $q$ from this collection. The factor $q$ may depend on both $B$ and $\lambda$; the theorem does not assert that one fixed factor works for all $\lambda$.

This conclusion is analogous to the random version of Szemerédi's theorem: above the natural threshold, with high probability every subset of fixed positive relative density inside a binomial random subset of $[N]$ contains a nontrivial arithmetic progression of the prescribed length; see \cite{Schacht16} and, independently, \cite{CG16}. In the present setting, the relevant configuration is a pair of points at a prescribed squared-distance scale.

To formally state our theorem we need the following notation. For $N\in\N$ and $\lambda\in\N$, let $Q_N=[-N,N]^d\cap \Z^d$, and $S_\lambda=\{\b{v}\in \Z^d: |\b{v}|^2=\lambda\}$.
For $q\in\N$ and $f:\Z^d\to\R$, define the dilated spherical average
\begin{align}
    A_\lambda^{(q)}f(\b x )=\frac{1}{|S_\lambda|}\sum_{\b{v}\in S_\lambda} f(\b x +q\b{v}).
\end{align}
Finally, let $D^2(B)=\{|\b x -\b y |^2:\b x ,\b y \in B\}$ denote the squared-distance set of a set $B\subseteq \Z^d$.

\begin{mainthm}\label{mainthm:multi-count}
Let $d \geq 5$, let $0 < \gamma < d - 2$, and put $p_N = N^{-\gamma}$. Let $\Omega_N \subseteq Q_N$ be the binomial random subset obtained by retaining each point of $Q_N$ independently with probability $p_N$. For every $0 < \delta < 1$, there exist integers $s > r \geq 2$ and $J \geq 0$ depending only on $d$ and $\delta$, a finite set
\begin{align}
\mathcal Q
    &= \left\{ r^j s^{J-j} : 0 \leq j \leq J \right\},
\end{align}
and constants $c_0 = c_0(d,\delta) > 0$, $c_1 = c_1(d,\delta) > 0$, $c_2 = c_2(d,\delta) > 0$, and $\mathfrak c = \mathfrak c(d,\delta,\gamma) > 0$ such that, for all sufficiently large $N$, with probability at least $1 - \exp\left(-\mathfrak c N^{d-\gamma}\right)$, every subset $B \subseteq \Omega_N$ satisfying $|B| \geq \delta |\Omega_N|$ has the following property: for every integer $\lambda$ satisfying
\begin{align}
c_1 p_N^{-2/(d-2)}
    & \leq \lambda
    \leq c_2 N^2,
\end{align}
we have
\begin{align}
\frac{1}{|Q_N|} \sum_{\b{x} \in \Z^d} \one_B(\b{x}) \sum_{j=0}^{J} A_{s^{2(J-j)}\lambda}^{(r^j)} \one_B(\b{x})
    &\geq c_0 p_N^2.
\label{eq:thm-multi-count}
\end{align}
Moreover, if $\rho_N$ denotes the probability that the preceding conclusion fails, then 
$|\log \rho_N| \simeq_{d,\delta,\gamma} N^{d-\gamma},$
so the failure-probability estimate is optimal at the exponential scale.
\end{mainthm}

\begin{crlr}\label{maincrlr:multi-distances}
For all sufficiently large $N$, with probability at least $1-\exp(-\mathfrak c N^{d-\gamma})$, every subset $B\subseteq\Omega_N$ satisfying $|B|\geq\delta|\Omega_N|$ satisfies
\begin{align}
D^2(B)\cap \left\{
    r^{2j}s^{2(J-j)}\lambda:
    0\leq j\leq J \right\}
\neq\varnothing
\end{align}
for every integer $\lambda$ in the above range, or equivalently, $D^2(B)\cap\left\{q^2\lambda:q\in\mathcal Q\right\}\neq\varnothing.$
\end{crlr}

For any sequence of positive integers satisfying $\lambda_N = o\big(p_N^{-2/(d - 2)}\big)$, the pairs at the finitely many squared distances $q^2 \lambda_N$, $q \in \mathcal Q$, can, with high probability, all be destroyed by removing a vanishing proportion of $\Omega_N$. The following proposition therefore shows that the power-law dependence in the lower endpoint of Theorem \ref{mainthm:multi-count} is sharp.

\begin{mainprop}\label{prop:lower-sharp}
Under the hypotheses of Theorem \ref{mainthm:multi-count}, let $r$, $s$, $J$, and $\mathcal Q$ be the parameters supplied there. Let $(\lambda_N)_{N \in \N}$ be any sequence of positive integers satisfying $\lambda_N = o\big(p_N^{-2/(d - 2)}\big)$. Then there exists a deterministic sequence $\varepsilon_N \to 0$ such that, with probability tending to $1$, there is a set $B_N \subseteq \Omega_N$ satisfying $|B_N| \geq (1 - \varepsilon_N)|\Omega_N|$ and
\begin{align}
D^2(B_N) \cap \left\{ q^2 \lambda_N : q \in \mathcal Q \right\}
    &= \varnothing.
\end{align}
\end{mainprop}

Corollary \ref{maincrlr:multi-distances} assigns to each admissible integer $\lambda$ a dilation $q_\lambda \in\mathcal Q$ such that $q_\lambda ^2\lambda\in D^2(B)$, thereby defining a finite coloring of the admissible interval. Van der Waerden's theorem then supplies a sufficiently long monochromatic arithmetic progression, from which we obtain an affine copy $a+tF\subseteq D^2(B)$ of every fixed finite set $F\subseteq\Z$. Specifically,

\begin{mainthm}\label{mainthm:sparse-random-patterns}
Under the hypotheses of Theorem \ref{mainthm:multi-count}, for every fixed finite set $F\subseteq\Z$ with $|F|\geq2$ and all sufficiently large $N$, with probability at least $1-\exp(-\mathfrak c N^{d-\gamma})$, every subset $B\subseteq\Omega_N$ satisfying $|B|\geq\delta|\Omega_N|$ has the property that there exist $a\in\Z$ and $t\in\N$ such that
\begin{align}
a+tF\subseteq D^2(B).
\end{align}
\end{mainthm}

Taking $F=\{0,1,\ldots,k-1\}$ in Theorem \ref{mainthm:sparse-random-patterns} immediately shows that $D^2(B)$ contains at least one nontrivial arithmetic progression of length $k$. We are, however, able to make a stronger, quantitative statement about the number of arithmetic progressions, using pigeonholing and Varnavides' averaging form applied to  Szemer\'edi's theorem \cites{Varnavides59, Szemeredi75}.

\begin{mainthm}\label{mainthm:sparse-random-aps}
Under the hypotheses of Theorem \ref{mainthm:multi-count}, for every fixed integer $k\geq2$, there exists a constant $c_k =c_k (d,\delta)>0$ such that, for all sufficiently large $N$, with probability at least $1-\exp(-\mathfrak c N^{d-\gamma})$, every subset $B\subseteq\Omega_N$ with $|B|\geq\delta|\Omega_N|$ has the property that
\begin{align}
\Big| \big\{ (a,b) \in\Z^2 \,:\, a\geq 0, \; b\geq1\text{ and } a+tb\in D^2(B)\text{ for every }0\leq t\leq k-1 \big\}\Big|
\geq c_k N^4.
\end{align}
In particular, $D^2(B)$ contains at least $c_k N^4$ distinct nontrivial arithmetic progressions of length $k$.
\end{mainthm}

The pigeonholing argument used to prove Theorem \ref{mainthm:sparse-random-aps} also shows that, for each fixed $B$, one dilation $q_B \in \mathcal Q$ realizes a positive proportion of all admissible scales, yielding the following maximal-order estimate for $D^2(B)$.

\begin{crlr}\label{maincrlr:prevalent-dilation}
Under the hypotheses of Theorem \ref{mainthm:multi-count}, let
\begin{align}
I_N
    &= \left\{ \lambda \in \N : c_1 p_N^{-2/(d - 2)} \leq \lambda \leq c_2 N^2 \right\}.
\end{align}
For all sufficiently large $N$, with probability at least $1 - \exp(-\mathfrak c N^{d-\gamma})$, every subset $B \subseteq \Omega_N$ satisfying $|B| \geq \delta |\Omega_N|$ admits a dilation $q_B \in \mathcal Q$ such that
\begin{align}
\left| \left\{ \lambda \in I_N : q_B^2 \lambda \in D^2(B) \right\} \right|
    &\geq \frac{|I_N|}{|\mathcal Q|}.
\end{align}
In particular, $|D^2(B)| \simeq_{d,\delta} N^2.$
\end{crlr}

\subsection{Proof Outline}

The proof has two main components. We first establish a deterministic supersaturation result for dense subsets of $Q_N$, and then transfer it to subsets of positive relative density inside the sparse random set $\Omega_N$ using the graph case of Schacht's theorem \cite{Schacht16}.

\emph{The density argument.}
Proposition \ref{prop:gvn} gives a generalized von Neumann estimate for spherical averages when the balance function of a set is sufficiently uniform. Because a single prescribed dilation is obstructed by congruence phenomena, as shown in Example \ref{ex:obstruction}, we work with a finite family of dilations. If the required uniformity fails, Lemma \ref{lm:density-incr} produces a density increment on a smaller affine copy of the lattice. Iterating this dichotomy over finitely many scales gives the local counting statement in Lemma \ref{lm:dense-lm-local}. Averaging the local count over translates then yields the global supersaturation estimate in Proposition \ref{prop:dense-lm-supersat}.

\emph{Sparse transference.}
For each admissible $\lambda$, the relevant spherical pairs are encoded by a graph on $Q_N$. Proposition \ref{prop:dense-lm-supersat} verifies the density condition required by Schacht's theorem, while Lemma \ref{lm:graph-bddness} and Corollary \ref{cor:graph-bddness-cnst} provide the necessary second-moment bounds. Since $|S_\lambda|\simeq_d\lambda^{(d-2)/2}$, the diagonal term is controlled once $\lambda\gtrsim p_N^{-2/(d-2)}$, after choosing the implicit constant sufficiently large. At this same endpoint, Schacht's divergence quantity is bounded below by a fixed multiple of $p_N|Q_N|\simeq_d N^{d-\gamma}$ and therefore tends to infinity. The family contains only $O(N^2)$ graphs, so the corresponding union bound does not introduce any logarithmic loss. Lemma \ref{lm:quant-schacht-graphs} consequently gives Proposition \ref{prop:sparse-rel} throughout the stated power-law range.

\emph{Completion and consequences.}
A Chernoff bound shows that $|\Omega_N|$ is comparable to $p_N|Q_N|$, so Proposition \ref{prop:sparse-rel} immediately implies Theorem \ref{mainthm:multi-count} and Corollary \ref{maincrlr:multi-distances}. 
Coloring each admissible $\lambda$ by a dilation $q\in\mathcal Q$ for which $q^2\lambda\in D^2(B)$ and applying van der Waerden's theorem gives Theorem \ref{mainthm:sparse-random-patterns}. Pigeonholing one dilation and applying the Varnavides averaging form of Szemerédi's theorem \cites{Varnavides59,Szemeredi75} gives the quantitative arithmetic-progression conclusion in Theorem \ref{mainthm:sparse-random-aps}. The same pigeonholing step yields Corollary \ref{maincrlr:prevalent-dilation}, showing that one dilation realizes a positive proportion of the admissible scales and, consequently, that $|D^2(B)|\simeq_{d,\delta}N^2$.

\emph{Sharpness.}
When $\lambda_N=o\big(p_N^{-2/(d-2)}\big)$, the expected number of pairs in $\Omega_N$ at the finitely many relevant distance scales is $o(p_N|Q_N|)$. Removing one endpoint from each such pair leaves, with high probability, a subset of density $1-o(1)$ containing none of those distances, proving Proposition \ref{prop:lower-sharp}. Finally, the event $\Omega_N=\varnothing$ has probability $\exp\big(-(2^d+o(1))N^{d-\gamma}\big)$ and forces the conclusion of Theorem \ref{mainthm:multi-count} to fail. Together with the transference upper bound, this gives the optimal exponential scale of the failure probability.

\bigskip
\section{Preliminaries}
\subsection{Notation}
We use the standard notation $a\lesssim b$ to imply that there exists $C>0$ such that $a\leq Cb$. When $C$ depends on a parameter $\kappa$, we shall write $a\lesssim_\kappa b$. Similarly, we write $a\simeq b$ when $a\lesssim b$ and $b\lesssim a$, simultaneously.

Let $\sigma_\lambda=|S_\lambda|^{-1}\one_{S_\lambda}$ denote the normalized discrete spherical measure. 
For a cube $[-L/2,L/2]^d$ define the normalized indicator in scale $R$ as
\begin{align}
    \chi_{R,L}(\b x )=
        \left(\dfrac{R}{L}\right)^d \one_{(R\Z)^d\cap[-L/2,L/2]^d} (\b x).
\end{align}

For a fixed $\eta \in (0,1)$, define $q_\eta=\lcm\{1\leq q\leq c\eta^{-2}\}$ for some positive constant $c$ that depends on the individual problem.

\subsection{Auxiliary Results}

Magyar, Stein and Wainger \cite{MSW2002} developed the Fourier-analytic theory of discrete spherical averages in dimensions $d\geq5$, using the circle method to decompose the discrete spherical measure into contributions localized near rational frequencies together with an error term satisfying strong decay estimates. Their work has become a fundamental tool in the study of discrete spherical averages and is closely related to the Fourier-analytic input needed here. For our purposes, it is more convenient to use the following related pointwise estimate of Magyar \cite{M2008}, which directly captures the decay of the Fourier transform of the spherical measure away from rational frequencies with bounded denominator.

\begin{lemma}\label{lm:magyar}
Let $d\geq5$ and $0<\eta<1$. There exists a constant $C=C(d)>0$ such that, if
$q_\eta=\lcm\{1\leq q\leq C\eta^{-2}\}$, $\lambda\geq C\eta^{-4}$, and $\btheta\in\T^d$ satisfies
\begin{align}
\dist_{\T^d}\left(\btheta,q_\eta^{-1}\Z^d\right)
    &> \eta^{-1/2}\lambda^{-1/2},
\end{align}
then $|\hat\sigma_\lambda(\btheta)| \leq\eta.$
\end{lemma}

We shall also need an adapted generalized von Neumann inequality for spherical averages, following an approach similar to \cite{LM2020}.

\begin{prop}\label{prop:gvn}
Let $d\geq 5$ and let $0<\eta<1$. There is a constant $c_0=c_0(d)\geq 1$ with the following property. Let $r,L,N,\lambda\in\N$, set $R=rq_\eta$, with $q_\eta=\lcm\{1\leq q\leq c_0\eta^{-2}\}$, and assume $1\leq R\leq \eta^2L\leq L\leq N$ and $\eta^{-4}L^2\leq \lambda\leq \eta^4N^2$. Let $\chi_{R,L}(\b x )$ be the normalized indicator of $[-L/2,L/2]^d$ in scale $R$.
For $f:Q_N\to\R$, define the relative uniformity norm as
\begin{align}
\|f\|_{U^1(R,L)} 
    & :=\left(\frac{1}{|Q_N|}\sum_{\b t \in\Z^d}|f*\chi_{R,L}(\b t )|^2\right)^{1/2}.
\end{align}
Then for all functions $f_0,f_1:Q_N\to[-1,1]$ we have
\begin{align}
    \left|\frac{1}{|Q_N|}\sum_{\b x \in\Z^d}f_0(\b x )A_\lambda^{(r)}f_1(\b x )\right|\leq \|f_1\|_{U^1(R,L)}+O_{d,r}(\eta).
\end{align}
\end{prop}

\begin{proof}
Choose $c_0$ at least as large as the constant $C(d)$ in Lemma \ref{lm:magyar}.
It is enough to consider $\eta$ sufficiently small, since otherwise the result follows from the trivial bound by $1$.
Extend $f_0$ and $f_1$ by zero outside $Q_N$. By Cauchy-Schwarz and Plancherel,
\begin{align}
\left| \frac{1}{|Q_N|} \sum_{\b x \in\Z^d} 
    f_0(\b x ) A_\lambda^{(r)}f_1(\b x ) \right|^2
&\leq \frac{1}{|Q_N|}   \sum_{\b x \in\Z^d} 
    |A_\lambda^{(r)} f_1(\b x )|^2 \\
&=\frac{1}{|Q_N|} \int_{\T^d}| \hat f_1(\bxi)|^2 
    |\hat\sigma_\lambda(r\bxi)|^2\,d\bxi.
\end{align}
Notice that $R\leq\eta^2L$ and $R\geq1$ imply $L\geq\eta^{-2}$, and hence $\lambda\geq\eta^{-4}L^2\geq\eta^{-8}$.

We claim that, uniformly in $\bxi\in\T^d$,
\begin{align}
    |\hat\sigma_\lambda(r\bxi)|^2
    \leq
    |\hat\chi_{R,L}(\bxi)|^2+O_{d,r}(\eta^2).
    \label{eq:gvn-multiplier-bound}
\end{align}
Set $\rho=\eta^{-1/2}\lambda^{-1/2}$. If
$\dist_{\T^d} \left( r\bxi, q_\eta^{-1} \Z^d \right) > \rho$,
then \eqref{eq:gvn-multiplier-bound} follows directly from Lemma \ref{lm:magyar}.
Otherwise, there exists $\boldsymbol{\ell}\in\Z^d$ such that
$\dist_{\T^d} \left(\bxi, R^{-1}\boldsymbol{\ell} \right) \leq \rho/r.$
Choose $\btheta\in\R^d$ with $\bxi=\boldsymbol{\ell}/R+\btheta$ in $\T^d$ and $|\btheta|\leq\rho/r$. Put
\begin{align}
m_{R,L}
    = \sum_{\b y \in\Z^d}\chi_{R,L}(\b y )
    = 1+O_d\left(\frac{R}{L}\right)
    = 1+O_d(\eta^2).
\end{align}
Since $\chi_{R,L}$ is symmetric and every $\b y $ in its support belongs to $(R\Z)^d$, we have
\begin{align}
\hat\chi_{R,L}(\bxi)
=\sum_{\b y  \in\Z^d} \chi_{R,L}
    (\b y )\cos(2\pi\b y \cdot\btheta)
\geq m_{R,L} - C_dL^2|\btheta|^2m_{R,L}.
\end{align}
Moreover,
$L^2|\btheta|^2 \leq (\eta\lambda)^{-1} r^{-2} L^2
    \leq \eta^3 r^{-2}
    \leq \eta^2. $
Consequently, $|\hat\chi_{R,L}(\bxi)|
    \geq 1-O_{d,r}(\eta^2).$
Since $|\hat\sigma_\lambda(r\bxi)|\leq1$, this proves \eqref{eq:gvn-multiplier-bound}.

Inserting \eqref{eq:gvn-multiplier-bound} into the Plancherel estimate gives
\begin{align}
\left| \frac{1}{|Q_N|} 
    \sum_{\b x \in\Z^d} 
    f_0(\b x )A_\lambda^{(r)} 
    f_1(\b x )\right|^2
&\leq
    \frac{1}{|Q_N|}
    \int_{\T^d}
    |\hat f_1(\bxi)|^2 
    |\hat\chi_{R,L}(\bxi)|^2\,d\bxi \\ & \hspace*{2cm}
    +O_{d,r}(\eta^2)\frac{1}{|Q_N|}\int_{\T^d}|\hat f_1(\bxi)|^2\,d\bxi \\
    &=\|f_1\|_{U^1(R,L)}^2
    +O_{d,r}(\eta^2)\frac{1}{|Q_N|}\sum_{\b x \in\Z^d}|f_1(\b x )|^2 \\
    &\leq \|f_1\|_{U^1(R,L)}^2+O_{d,r}(\eta^2).
\end{align}
\end{proof}

\bigskip 
\section{A Density Increment Argument}
\subsection{Multiple Scales}
We require a finite dense consequence of the Lyall-Magyar density increment argument from \cite{LM2020}. One important point is that a single prescribed dilation cannot give the required counting estimate uniformly over all dense sets and scales. Instead, one fixes a finite collection of dilations and a common base scale $\lambda$.

\begin{prop}\label{prop:dense-lm}
Let $d\geq 5$, $0<\delta<1$, and $0<\varepsilon<1$. There exists $0<\eta=\eta(d,\delta,\varepsilon)<1/2$ such that the following holds. Let $L,N\in\N$ satisfy $\eta^{-2}q_\eta \leq L \leq \eta^4N,$ where $q_\eta=\lcm\{1\leq q\leq c_0\eta^{-2}\}$ and $c_0=c_0(d)$ is the constant in Proposition~\ref{prop:gvn}. Let $B\subseteq Q_N$ satisfy $|B|\geq\delta|Q_N|$, $\alpha=|B|/|Q_N|$ be the relative density of $B$ in $Q_N$ and $f_B=\one_B-\alpha\one_{Q_N}$ the associated balance function. If $\|f_B\|_{U^1(q_\eta,L)}\leq 2\eta$, we have
\begin{align}
\frac{1}{|Q_N|} \sum_{\b x\in\Z^d} \one_B(\b x)
    A_\lambda^{(1)}\one_B(\b x)
& \geq (1-\varepsilon)\delta^2,
\label{eq:dense-lm}
\end{align}
for every integer $\lambda$ satisfying $\eta^{-4}L^2\leq\lambda\leq\eta^4N^2$.\\
Consequently, $\lambda\in D^2(B)$ for every integer
$\lambda$ in this range.
\end{prop}

\begin{proof}
Proposition~\ref{prop:gvn}, applied with $r=1$, gives
\begin{align}
\left| \frac{1}{|Q_N|} \sum_{\b x\in\Z^d}
    \one_B(\b x)A_\lambda^{(1)}f_B(\b x) \right|
& \lesssim_d\eta.
\end{align}
Moreover, the only points $\b x\in Q_N$ for which $A_\lambda^{(1)}\one_{Q_N}(\b x)\neq 1$ lie in a boundary layer of thickness at most $\sqrt{\lambda}$. Hence
\begin{align}
\frac{1}{|Q_N|} \sum_{\b x\in\Z^d}
    \one_B(\b x)A_\lambda^{(1)}\one_{Q_N}(\b x)
& =\alpha+O_d\left(\frac{\sqrt{\lambda}}{N}\right)
=\alpha+O_d(\eta^2),
\end{align}
where the last estimate follows from the upper bound in the stated range.
Combining the preceding estimates with $\one_B=\alpha\one_{Q_N}+f_B$, we get 
\begin{align}
\frac{1}{|Q_N|} \sum_{\b x\in\Z^d} 
    \one_B(\b x)A_\lambda^{(1)}\one_B(\b x)
& =\alpha^2+O_d(\eta).
\end{align}
Choose $\eta<1/2$ sufficiently small that the absolute value of the error term is at most $\varepsilon\delta^2$. Since $\alpha\geq\delta$, this yields \eqref{eq:dense-lm}.

It remains to prove the distance conclusion. Since $(1-\varepsilon)\delta^2>0$ and all summands in \eqref{eq:dense-lm} are nonnegative, there exist $\b x\in B$ and $\b{v}\in S_\lambda$ such that $\b y=\b x+\b{v}\in B$. Consequently, $|\b x-\b  y|^2=|\b{v}|^2=\lambda,$ so $\lambda\in D^2(B)$.
\end{proof}

\begin{rmk}
We retain $\varepsilon$ in Proposition~\ref{prop:dense-lm} to record the dependence of the uniformity parameter and the admissible scales on the permitted loss in the counting estimate. In all subsequent applications we take $\varepsilon=1/2$ and suppress dependence on this fixed choice.
\end{rmk}
\smallskip 

It is tempting to hope for a dense finite Lyall-Magyar input with one prescribed dilation: for fixed $d \geq 5$ and $0 < \delta < 1$, there would exist $q, \Lambda_0 \in \N$ and $\zeta, c > 0$, depending only on $d$ and $\delta$, such that, for all sufficiently large $N$, and $B \subseteq Q_N$ with $|B| \geq \delta |Q_N|$ we have
\begin{align}
    \frac{1}{|Q_N|} \sum_{\b x \in \Z^d} \one_B(\b x) A_\lambda^{(q)} \one_B(\b x)
    \geq \zeta \label{eq:false-fixed-dilate-lm}
\end{align}
for every integer $\lambda$ with $\Lambda_0 \leq \lambda \leq c N^2$. The following example rules this out when $0 < \delta < 1/2$, motivating our use below of a finite family of dilations at a common base scale $\lambda$.

This obstruction is consistent with Proposition \ref{prop:dense-lm}, whose conclusion requires the uniformity condition $\|f_B\|_{U^1(q_\eta, L)} \leq 2 \eta$. When this condition fails, the density-increment argument passes to an affine sublattice on which the relative density increases; returning to the original set introduces an additional dilation. Iterating over a finite sequence of scales produces the family of dilations in Lemma \ref{lm:dense-lm-local} and Proposition \ref{prop:dense-lm-supersat}. Likewise, Corollary \ref{maincrlr:prevalent-dilation} selects $q_B \in \mathcal Q$ after $B$ is fixed and requires it to realize only a positive proportion of the admissible scales. The example instead rules out one dilation, chosen independently of $B$, satisfying the counting estimate \eqref{eq:false-fixed-dilate-lm} at every admissible scale.

\begin{example}[A periodic obstruction to one fixed dilation]\label{ex:obstruction}
Fix $d\geq 5$, $0<\delta<1/2$, $q\in\N$, $c>0$, and $\Lambda_0\in\N$. For $\b x \in\Z^d$ and $1\leq i\leq d$, let $r_i(\b x )$ be the unique integer satisfying both $0\leq r_i(\b x )<2q$ and $r_i(\b x )\equiv x_i \pmod {2q}.$
Define
\begin{align}
& \epsilon_i(\b x )
    =\left\lfloor \frac{r_i(\b x )}{q}\right\rfloor\in\{0,1\},
&& \text{and}
& \pi_q(\b x )
    \equiv \sum_{i=1}^d \epsilon_i(\b x ) \pmod 2.
\end{align}
Let $B_N=\{\b x \in Q_N:\pi_q(\b x )\equiv 0 \pmod 2\}.$
Among the $(2q)^d$ residue classes modulo $(2q\Z)^d$, exactly one half satisfy the parity condition  for $B_N$, hence
\begin{align}
    |B_N|=\frac{1}{2}|Q_N|+O_{d,q}(N^{d-1}).
\end{align}
Since $\delta<1/2$, it follows that $|B_N|\geq \delta |Q_N|$ for all sufficiently large $N$. Choose $N$ sufficiently large so that the interval $[\Lambda_0,cN^2]$ contains an odd positive integer $\lambda$. Since $d\geq 5$, the sphere $S_\lambda$ is nonempty.

We claim that no point of $B_N$ has a $qS_\lambda$-neighbor in $B_N$. Indeed, for every $\b x \in\Z^d$ and every $\b{v}\in\Z^d$, we have
$\epsilon_i(\b x +q\b{v})\equiv \epsilon_i(\b x )+v_i \pmod 2$
for each $1\leq i\leq d$. Summing over $i$ gives
$\pi_q(\b x +q\b{v})-\pi_q(\b x )\equiv \sum_{i=1}^d v_i\equiv \sum_{i=1}^d v_i^2\equiv |\b{v}|^2 \pmod 2.$
If $\b{v}\in S_\lambda$ and $\lambda$ is odd, then this gives
$\pi_q(\b x +q\b{v})-\pi_q(\b x )\equiv 1 \pmod 2.$
Thus $\b x $ and $\b x +q\b{v}$ have opposite $\pi_q$-parity. In particular,
$ \one_{B_N}(\b x )\one_{B_N}(\b x +q\b{v})=0$
for every $\b x \in\Z^d$ and every $\b{v}\in S_\lambda$. Therefore,
\begin{align}
    \frac{1}{|Q_N|}\sum_{\b x \in\Z^d}\one_{B_N}(\b x )A_\lambda^{(q)}\one_{B_N}(\b x )&=\frac{1}{|Q_N|\,|S_\lambda|}\sum_{\b x \in\Z^d}\sum_{\b{v}\in S_\lambda}\one_{B_N}(\b x )\one_{B_N}(\b x +q\b{v})
    =0.
\end{align}
Thus no positive constant $\zeta$ can make \eqref{eq:false-fixed-dilate-lm} true uniformly for all dense subsets $B\subseteq Q_N$ when $0<\delta<1/2$.
\end{example}
\smallskip

\subsection{Building up to the General Case}
\begin{lemma} \label{lm:density-incr}
Let $d\geq 5$ and $0<\rho<1$. There exist an integer $r \geq 2$ and constants $\tau>0$, $0<c\leq1$, $K_0\geq1$, $C_0=C_0(d,r)\geq1$, and $\lambda_0\in\N$, depending only on $d$ and $\rho$, with the following property. Let $\mu,N\in\N$ such that $\mu\geq\lambda_0$ and $N\geq K_0\sqrt{\mu}$,
and let $V\subseteq Q_N$ have density
$\alpha=|V|/|Q_N|\geq\rho$.
Then at least one of the following holds:
\begin{enumerate}[label=\textnormal{(\roman*)}]
    \item $\displaystyle \sum_{\b x \in\Z^d}\one_V(\b x )A_\mu^{(1)}\one_V(\b x )
    \geq \frac{\rho ^2}2 |Q_N|$ , \\[-0.2cm]
    \item there exist an integer $N'\geq c\sqrt\mu/r-C_0$, a vector $\b{a}\in\Z^d$, and a set $V'\subseteq Q_{N'}$ such that $\b{a}+rV'\subseteq V$ and $|V'|/|Q_{N'}|\geq\alpha(1+\tau)$.
\end{enumerate}
\end{lemma}

\begin{proof}
Apply Proposition \ref{prop:dense-lm} with $\delta=\rho$. Let $0<\eta<1/2$ be the parameter supplied by that proposition and set $r=q_\eta.$
Choose $K_0\geq\eta^{-2}$ sufficiently large, depending only on $d$ and $\eta$, and choose $\lambda_0$ sufficiently large that all estimates below hold whenever $\mu\geq\lambda_0$. Define
$ L=\lfloor\eta^2\sqrt{\mu}/4\rfloor.$
We have
$ \eta^{-2}r\leq L,$ $\eta^{-4}L^2\leq\mu,$ $\mu\leq\eta^4N^2,$ and $ L\leq \eta^4N,$ where the third inequality follows from $N\geq K_0\sqrt{\mu}$ and $K_0\geq\eta^{-2}$.

Let $ f_V=\one_V-\alpha\one_{Q_N}$ be the balance function of $V$ in $Q_N$.
Suppose first that $ \|f_V\|_{U^1(r,L)}\leq2\eta.$
Proposition \ref{prop:dense-lm} immediately produces the first inequality.

It remains to consider the case  $\|f_V\|_{U^1(r,L)}>2\eta.$
Write $g(\b t )=(f_V*\chi_{r,L})(\b t ).$
Then
\begin{align}
    \sum_{\b t \in\Z^d}|g(\b t )|^2>4\eta^2|Q_N|.
\end{align}
Moreover,
\begin{align}
\sum_{\b t \in\Z^d}g(\b t )
    & = \Bigg(\sum_{\b y \in\Z^d} \chi_{r,L}(\b y )\Bigg)
    \Bigg(\sum_{\b x \in\Z^d}f_V(\b x )\Bigg)
    =0.
\end{align}
The total mass of $\chi_{r,L}$ lies between $1/2$ and $2$, and hence $|g(\b t )|\leq2$. Therefore
\begin{align}
    \sum_{\b t :g(\b t )>0}g(\b t )
    =\frac12\sum_{\b t \in\Z^d}|g(\b t )|
    \geq\frac14\sum_{\b t \in\Z^d}|g(\b t )|^2
    >\eta^2|Q_N|.
\end{align}

Let $\mathcal I$ be the set of $\b t $ for which $\b t -\big((r\Z)^d\cap[-L/2,L/2]^d\big)\subseteq Q_N.$
The support of $g$ outside $\mathcal I$ is contained in a boundary layer of cardinality at most
$ C_d\frac{L}{N}|Q_N|.$
Since $L/N\leq\eta^2/(4K_0)$, choosing $K_0$ sufficiently large gives
\begin{align}
    \sum_{\substack{\b t \notin\mathcal I\\g(\b t )>0}}g(\b t )
    \leq\frac12\eta^2|Q_N|
\qquad \Longrightarrow \qquad 
    \sum_{\substack{\b t \in\mathcal I\\g(\b t )>0}}g(\b t )
    \geq\frac12\eta^2|Q_N|.
\end{align}
Since the support of $g$ has cardinality $O_d(|Q_N|)$, there exists $\b t _0\in\mathcal I$ such that $g(\b t _0)\geq c_d\eta^2$
for some $c_d>0$.

Let $N'$ be the integer for which $\{\b{z}\in\Z^d:r\b{z}\in[-L/2,L/2]^d\}=Q_{N'}.$
Then
\begin{align}
    N'\geq\frac{L}{2r}-O_{d,r}(1)
    \geq\frac{\eta^2}{8r}\sqrt{\mu}-O_{d,r}(1).
    \label{eq:fixed-scale-N-prime}
\end{align}
Define $ W=\{\b{z}\in Q_{N'}:\b t _0-r\b{z}\in V\}.$
Because $\b t _0\in\mathcal I$, all the points $\b t _0-r\b{z}$ with $\b{z}\in Q_{N'}$ belong to $Q_N$. Hence
$g(\b t _0)
    =(r/L)^d\left(|W|-\alpha|Q_{N'}|\right).$
Furthermore,
$(r/L)^d|Q_{N'}|=1+O_d(r/L)$
and since $g(\b t _0)\geq c_d\eta^2$, a further increase of $\lambda_0$ implies
\begin{align}
    \frac{|W|}{|Q_{N'}|}
    \geq\alpha+c_d'\eta^2
    \geq\alpha(1+c_d'\eta^2).
\end{align}
Set $\tau=c_d'\eta^2$ and $c=\eta^2/8$.
Finally, let $V'=-W$ and $\b{a}=\b t _0$. Since $Q_{N'}=-Q_{N'}$, we have $V'\subseteq Q_{N'}$, $\b{a}+rV'\subseteq V,$ and $|V'|/|Q_{N'}|\geq\alpha(1+\tau)$,
while \eqref{eq:fixed-scale-N-prime} gives
$ N'\geq c\sqrt{\mu}/{r}-O_{d,r}(1).$
This proves the alternative condition and concludes our proof.
\end{proof}

The preceding lemma gives a one-step dichotomy at a fixed scale: either the desired spherical count is already large, or one passes to a smaller affine copy on which the density increases by a definite factor. Iterating this alternative over a carefully chosen finite sequence of scales forces the counting outcome to occur before the density can exceed $1$. The next lemma packages this iteration into a local counting statement involving a finite family of dilations.

\begin{lemma}\label{lm:dense-lm-local}
Let $d\geq 5$ and $0<\rho<1$. There exist integers $s>r\geq 2$ and $J\geq 0$, and constants $K\geq 1$, $\kappa>0$, and $\lambda_0\in\N$, depending only on $d$ and $\rho$, with the following property. For every integer $\lambda\geq\lambda_0$, every translate $C=\b t +Q_M$ with $M=\left\lceil Ks^J\sqrt{\lambda}\right\rceil$, and every set $U\subseteq C$ satisfying $|U|\geq\rho|Q_M|$, we have
\begin{align}
\sum_{j=0}^{J} \sum_{\b x \in\Z^d}
    \one_U(\b x )
    A_{s^{2(J-j)}\lambda}^{(r^j)}
    \one_U(\b x )
& \geq \kappa|Q_M|.
\label{eq:dense-lm-local-count}
\end{align}
Consequently, $D^2(U)\cap \big\{ r^{2j}s^{2(J-j)}\lambda \,:\,  0\leq j\leq J \big\} \neq\varnothing$.
\end{lemma}

\begin{proof}
By translation invariance, it is enough to consider $C=Q_M$. Apply Lemma \ref{lm:density-incr} with parameters $d$ and  $\rho$, and let $r$, $\tau$, $c$, $K_0$, and $\lambda_1$ be the resulting constants. Fix $C_0=C_0(d,r)\geq1$ such that the side-length conclusion in alternative (ii) of that lemma may always be written as
$N'\geq\frac{c\sqrt{\mu}}{r}-C_0.$
Choose $J\geq0$ so that
\begin{align}
    \rho(1+\tau)^{J+1}>1.
\end{align}
Next choose an integer $s>r$ sufficiently large that $2rK_0\leq cs$,
and set $K=K_0$. Finally, choose $\lambda_0\geq\lambda_1$ sufficiently large that, whenever $\lambda\geq\lambda_0$,
\begin{align}
    C_0\leq\frac{cs}{2r}\sqrt{\lambda}.
    \label{eq:dense-lm-local-lambda0}
\end{align}

Fix $\lambda\geq\lambda_0$, 
and let $U\subseteq Q_M$ have density $\alpha_0=|U|/|Q_M|\geq\rho$. We inductively construct integers $M_j\geq1$, sets $U_j\subseteq Q_{M_j}$, and affine maps $\phi_j(\b{z})=\b{a}_j+r^j\b{z}$
such that $\phi_j(U_j)\subseteq U$,
$\alpha_j:=\frac{|U_j|}{|Q_{M_j}|}\geq\rho(1+\tau)^j$, and 
\begin{align} 
M_j\geq K_0 s^{J-j}\sqrt{\lambda}.
\label{eq:dense-lm-local-induction}
\end{align}
At stage $j=0$, take $M_0=M$, $U_0=U$, and $\phi_0$ to be the identity. Since $K=K_0$, \eqref{eq:dense-lm-local-induction} still holds.

Suppose that the objects at stage $j$, where $0\leq j\leq J$, have been constructed. Since $s^{2(J-j)}\lambda \geq\lambda\geq\lambda_1$ and $M_j\geq K_0s^{J-j}\sqrt{\lambda}$, Lemma \ref{lm:density-incr} applies to $U_j\subseteq Q_{M_j}$. If its counting alternative holds, then
\begin{align}
    \sum_{\b{z}\in\Z^d}
    \one_{U_j}(\b{z})
    A_{s^{2(J-j)}\lambda}^{(1)}
    \one_{U_j}(\b{z})
    \geq
    \frac{\rho^2}2|Q_{M_j}|,
    \label{eq:dense-lm-local-terminal}
\end{align}
and the iteration stops.

Otherwise, the density-increment alternative supplies an integer $M_{j+1}$, a vector $\b{b}_j\in\Z^d$, and a set $U_{j+1}\subseteq Q_{M_{j+1}}$ such that
\begin{align}
U_j & \supseteq \b{b}_j+rU_{j+1},\\
\alpha_{j+1}& \geq\alpha_j(1+\tau), \\
M_{j+1}& \geq\frac{c}{r}s^{J-j}\sqrt{\lambda}-C_0.
\end{align}
Define $\phi_{j+1}(\b{z}) = \phi_j(\b{b}_j+r\b{z}) = \b{a}_{j+1}+r^{j+1}\b{z}.$
Then $\phi_{j+1}(U_{j+1})\subseteq U$, and the required density estimate at stage $j+1$ follows from the preceding inequalities. If $j<J$, then the choices of $s$ and $\lambda_0$ give
\begin{align}
M_{j+1} &\geq \frac{c}{r}s^{J-j}\sqrt{\lambda}-C_0 \\
    & \geq \frac{c}{2r}s^{J-j}\sqrt{\lambda} \\
    & = \frac{cs}{2r}s^{J-j-1}\sqrt{\lambda} \\
    & \geq K_0s^{J-j-1}\sqrt\lambda,
\end{align}
so \eqref{eq:dense-lm-local-induction} continues to hold.

The density-increment alternative cannot occur at every stage $j=0,\ldots,J$, since it would produce a set $U_{J+1}$ satisfying
\begin{align}
\frac{|U_{J+1}|}{|Q_{M_{J+1}}|}
    \geq \rho(1+\tau)^{J+1} > 1,
\end{align}
contradicting the choice of $J$. Thus the counting alternative holds at some stage $0\leq j\leq J$.

For this terminal value of $j$, the affine form of $\phi_j$ and the inclusion $\phi_j(U_j)\subseteq U$ imply, for every $\b{z}\in\Z^d$,
\begin{align}
\one_U(\phi_j(\b{z}))
    A_{s^{2(J-j)}\lambda}^{(r^j)}
    \one_U(\phi_j(\b{z}))
& \geq \one_{U_j}(\b{z})
    A_{s^{2(J-j)}\lambda}^{(1)}
    \one_{U_j}(\b{z}).
\end{align}
Since $\phi_j$ is injective, summing over $\b{z}$ and using \eqref{eq:dense-lm-local-terminal} yields
\begin{align}
\sum_{\b x \in\Z^d} \one_U(\b x )
    A_{s^{2(J-j)}\lambda}^{(r^j)} \one_U(\b x )
& \geq \rho ^2 |Q_{M_j}|/2.
    \label{eq:dense-lm-local-pullback}
\end{align}

It remains to compare $|Q_{M_j}|$ with $|Q_M|$. If $j=0$, then $M_j=M$. If $1\leq j\leq J$, then stage $j$ arose from an increment at stage $j-1$, so \eqref{eq:dense-lm-local-lambda0} gives
\begin{align}
M_j &\geq \frac{c}{r}s^{J-j+1}\sqrt{\lambda}-C_0 \\
    &\geq \frac{c}{2r}s^{J-j+1}\sqrt{\lambda}.
\end{align}
Moreover, $M=\left\lceil Ks^J\sqrt{\lambda}\right\rceil
\leq 2Ks^J\sqrt{\lambda}.$
Hence, with
\begin{align}
\theta = \min\left\{ 1, \frac{c}{4rK}s^{1-J} \right\} >0,
\end{align}
we have $M_j\geq\theta M$ for every possible terminal stage $j$. Since $|Q_T|\simeq_d T^d$ for $T\geq1$, there exists a constant $c_1=c_1(d,\rho)>0$ such that $|Q_{M_j}|\geq c_1|Q_M|.$
All summands in \eqref{eq:dense-lm-local-count} are nonnegative, so \eqref{eq:dense-lm-local-pullback} gives
\begin{align}
\sum_{i=0}^{J}\sum_{\b x \in\Z^d} \one_U(\b x )
    A_{s^{2(J-i)}\lambda}^{(r^i)} \one_U(\b x )
& \geq \rho^2 c_1|Q_M|/2.
\end{align}
Thus \eqref{eq:dense-lm-local-count} holds with $\kappa=\rho^2 c_1/2$. Finally, the $j$th summand counts pairs whose squared distance is
$|r^j\b{v}|^2 = r^{2j}s^{2(J-j)}\lambda,$
for $\b{v}\in S_{s^{2(J-j)}\lambda},$
which proves the distance conclusion. Translating back proves the result for every $C=\b t +Q_M$.
\end{proof}

Lemma \ref{lm:dense-lm-local} provides the required count inside any sufficiently dense cube at the scale dictated by $\lambda$. To obtain a global statement for a dense subset of $Q_N$, we average this local conclusion over translates of such cubes. Since each ordered pair is contained in only a controlled number of translates, the resulting local counts combine to give the following supersaturation estimate.

\begin{prop}\label{prop:dense-lm-supersat}
Let $d\geq 5$ and $0<\delta<1$. There exist integers $s>r\geq 2$ and $J\geq 0$, and constants $\zeta>0$, $c>0$, and $\Lambda_0\in\N$, depending only on $d$ and $\delta$, with the following property. For all sufficiently large $N$, every set $B\subseteq Q_N$ satisfying $|B|\geq\delta|Q_N|$ satisfies
\begin{align}
    \frac{1}{|Q_N|}\sum_{\b x \in\Z^d}\one_B(\b x )\sum_{j=0}^{J}A_{s^{2(J-j)}\lambda}^{(r^j)}\one_B(\b x )\geq\zeta
    \label{eq:dense-lm-supersat}
\end{align}
for every integer $\lambda$ satisfying $\Lambda_0\leq\lambda\leq cN^2.$
Thus, $D^2(B)\cap\left\{r^{2j}s^{2(J-j)}\lambda:0\leq j\leq J\right\}\neq\varnothing$
for every integer $\lambda$ in this range.
\end{prop}

\begin{proof}
Apply Lemma \ref{lm:dense-lm-local} with parameters $d$ and $\rho=\delta/4$. Let $s>r\geq2$, $J\geq0$, $K\geq1$, $\kappa>0$, and $\lambda_1\in\N$ be the resulting constants.

Choose $0<\vartheta<1$, depending only on $d$, sufficiently small that whenever $M\leq\vartheta N$ and $N$ is sufficiently large, the family
\begin{align}
    \mathcal T_{N,M}
    =
    \left\{\b t \in\Z^d:(\b t +Q_M)\cap Q_N\neq\varnothing\right\}
\end{align}
satisfies
\begin{align}
    |Q_N|\leq|\mathcal T_{N,M}|\leq2|Q_N|.
    \label{eq:dense-lm-supersat-translates}
\end{align}
Such a choice is possible since $Q_N\subseteq\mathcal T_{N,M}$ and $|\mathcal T_{N,M}|=(1+O_d(M/N))|Q_N|$ when $M/N$ is small. Let $\Lambda_0=\lambda_1$, and $c=\left(4Ks^J/\vartheta\right)^{-2}$.

Let $N$ be sufficiently large, let $B\subseteq Q_N$ satisfy $|B|\geq\delta|Q_N|$, and let $\Lambda_0\leq\lambda\leq cN^2$. Set $M=\big\lceil Ks^J\sqrt{\lambda}\big\rceil.$
By the definition of $c$, we have $M\leq\vartheta N$ once $N$ is sufficiently large. Write $\mathcal T=\mathcal T_{N,M}$ and, for each $\b t \in\mathcal T$, define
\begin{align}
\alpha_{\b t } : = \frac{|B\cap (\b t +Q_M)|}{|Q_M|}.
\end{align}
Every point of $B$ belongs to exactly $|Q_M|$ translates $\b t +Q_M$, all of which are indexed by elements of $\mathcal T$. Therefore
\begin{align}
\sum_{\b t \in\mathcal T}|B\cap (\b t +Q_M)|
    = |B||Q_M|,
\end{align}
and hence, by \eqref{eq:dense-lm-supersat-translates},
\begin{align}
    \frac{1}{|\mathcal T|}\sum_{\b t \in\mathcal T}\alpha_{\b t }
    =
    \frac{|B|}{|\mathcal T|}
    \geq
    \frac{\delta}{2}.
    \label{eq:dense-lm-supersat-avg-density}
\end{align}

Let
\begin{align}
    \mathcal G
    =
    \left\{\b t \in\mathcal T:\alpha_{\b t }\geq\frac{\delta}{4}\right\}.
\end{align}
Since $0\leq\alpha_{\b t }\leq1$, equation \eqref{eq:dense-lm-supersat-avg-density} gives
\begin{align}
    \frac{\delta}{2}
    &\leq
    \frac{|\mathcal G|}{|\mathcal T|}
    +
    \left(1-\frac{|\mathcal G|}{|\mathcal T|}\right)\frac{\delta}{4},
\end{align}
and therefore
\begin{align}
    |\mathcal G|
    \geq
    \frac{\delta}{4}|\mathcal T|
    \geq
    \frac{\delta}{4}|Q_N|.
    \label{eq:dense-lm-supersat-good-translates}
\end{align}

For every $\b t \in\mathcal G$, the set $U_{\b t }=B\cap (\b t +Q_M)$ has density at least $\rho=\delta/4$ in $(\b t +Q_M)$. Lemma \ref{lm:dense-lm-local} therefore gives
\begin{align}
\sum_{j=0}^{J}\sum_{\b x \in\Z^d}
    \one_{U_{\b t }}(\b x )
    A_{s^{2(J-j)}\lambda}^{(r^j)}
    \one_{U_{\b t }}(\b x )
& \geq \kappa|Q_M|.
\end{align}
Summing over $\b t \in\mathcal G$ yields
\begin{align}
\kappa|\mathcal G||Q_M|
& \leq \sum_{\b t \in\mathcal G}
    \sum_{j=0}^{J}
    \sum_{\b x \in\Z^d}
    \one_{U_{\b t }}(\b x )
    A_{s^{2(J-j)}\lambda}^{(r^j)}
    \one_{U_{\b t }}(\b x ).
    \label{eq:dense-lm-supersat-sum-local}
\end{align}

For any fixed ordered pair $\b x ,\b y \in\Z^d$, the number of translates $\b t +Q_M$ containing both $\b x $ and $\b y $ is at most $|Q_M|$. Expanding the spherical averages in \eqref{eq:dense-lm-supersat-sum-local} therefore gives
\begin{align}
\sum_{\b t \in\mathcal G}
    \sum_{j=0}^{J}
    \sum_{\b x \in\Z^d}
    \one_{U_{\b t }}(\b x )
    A_{s^{2(J-j)}\lambda}^{(r^j)}
    \one_{U_{\b t }}(\b x )
& \leq |Q_M|
    \sum_{j=0}^{J}
    \sum_{\b x \in\Z^d}
    \one_B(\b x )
    A_{s^{2(J-j)}\lambda}^{(r^j)}
    \one_B(\b x ).
\end{align}
Combining this with \eqref{eq:dense-lm-supersat-sum-local}, canceling $|Q_M|$, and using \eqref{eq:dense-lm-supersat-good-translates}, we obtain
\begin{align}
    \frac{1}{|Q_N|}
    \sum_{j=0}^{J}
    \sum_{\b x \in\Z^d}
    \one_B(\b x )
    A_{s^{2(J-j)}\lambda}^{(r^j)}
    \one_B(\b x )
    \geq
    \frac{\kappa\delta}{4},
\end{align}
which is our desired result.

Finally, all summands in \eqref{eq:dense-lm-supersat} are nonnegative, so there exist $0\leq j\leq J$, $\b x \in B$, and $\b{v}\in S_{s^{2(J-j)}\lambda}$ such that $\b y =\b x +r^j\b{v}\in B.$
Consequently,
$|\b x -\b y |^2 = r^{2j}|\b{v}|^2 = r^{2j}s^{2(J-j)}\lambda,$
which proves the distance conclusion.
\end{proof}

\bigskip
\section{Sparse Relative Spherical Counting Input}

We follow Schacht's transference approach for sparse random extremal problems \cite{Schacht16}. In the present setting, the relevant configurations consist of pairs of lattice points whose difference lies on a prescribed discrete sphere, so only the $2$-uniform case of his general framework is needed. The deterministic input is a supersaturation statement showing that every sufficiently dense subset of $Q_N$ contains many such pairs. Schacht's theorem transfers this conclusion to subsets of positive relative density inside the binomial random set $Q_{N,p_N}$, provided that the configurations are sufficiently numerous and that their local counts satisfy an appropriate second-moment bound. The lemma below verifies precisely this boundedness condition in the range of parameters used in our context.

\begin{lemma}\label{lm:graph-bddness}
Let $d \geq 5$, let $0 < \gamma < d - 2$, put $p_N = N^{-\gamma}$, and fix $q \in \N$. For $\lambda \in \N$, let $G_{\lambda,N}^{(q)}$ be the directed graph with vertex set $Q_N$ and edge set
\begin{align}
E_{\lambda,N}^{(q)}
    & = \{(\b{x}, \b{y}) \in Q_N \times Q_N : \b{y} - \b{x} \in q S_\lambda\}.
\end{align}
There are constants $c_0 = c_0(d,q) > 0$, $c_1 = c_1(d,q) > 0$, and $c_2 = c_2(d,q) > 0$ such that, uniformly for all integers $\lambda$ satisfying
\begin{align}
c_1 p_N^{-2/(d - 2)}
    & \leq \lambda
    \leq c_2 N^2,
\end{align}
we have $|S_\lambda| \simeq_d \lambda^{d/2 - 1},$ as well as $|E_{\lambda,N}^{(q)}| \simeq_{d,q} |Q_N| |S_\lambda|, $ and $p_N |S_\lambda| \geq 1$.
Furthermore, for every $p_N \leq \theta \leq 1$, if $Q_{N,\theta}$ denotes the binomial random subset of $Q_N$ with retention probability $\theta$, then
\begin{align}
\mathbb E \sum_{\b{x} \in Q_N} \deg_{G_{\lambda,N}^{(q)}}(\b{x}, Q_{N,\theta})^2
    & \leq c_0 \theta^2 \frac{|E_{\lambda,N}^{(q)}|^2}{|Q_N|}.
\end{align}
\end{lemma}

\begin{proof}
Observe that the estimate of the size of the discrete sphere is merely the standard estimate for the number of lattice points on discrete spheres, see \cite{MSW2002}.

We write $\deg_{G_{\lambda,N}^{(q)}}(\b{x}, U)$ for the out-degree of $\b{x}$ into $U$ in the directed graph $G_{\lambda,N}^{(q)}$. Choose $c_2 = c_2(d,q) > 0$ sufficiently small so that $q \sqrt{\lambda} \leq N/4$ whenever $\lambda \leq c_2 N^2$. For instance, it is enough to take $c_2 \leq 1/(16 q^2)$. If $\lambda \leq c_2 N^2$, then for every $\b{x} \in Q_{\lfloor N/2 \rfloor}$ and every $\b{v} \in S_\lambda$ we have $\b{x} + q \b{v} \in Q_N$. Indeed, since $\b{x} \in Q_{\lfloor N/2 \rfloor}$ and $q \sqrt{\lambda} \leq N/4$,
\begin{align}
\|\b{x} + q \b{v}\|_\infty
    & \leq \|\b{x}\|_\infty + q \|\b{v}\|_\infty
    \leq \frac{N}{2} + q \sqrt{\lambda}
    \leq \frac{3N}{4}
    \leq N,
\end{align}
and hence $\b{x} + q \b{v} \in Q_N$.

Consequently, $ |E_{\lambda,N}^{(q)}| \geq |Q_{\lfloor N/2 \rfloor}| |S_\lambda| \gtrsim_d |Q_N| |S_\lambda|.$ 
The reverse inequality is immediate, since for each fixed $\b{x} \in Q_N$ there are at most $|S_\lambda|$ possible out-neighbors $\b{x} + q \b{v}$ with $\b{v} \in S_\lambda$. Thus
$|E_{\lambda,N}^{(q)}| \leq |Q_N| |S_\lambda|.$
Combining the two gives $|E_{\lambda,N}^{(q)}| \simeq_{d,q} |Q_N| |S_\lambda|.$

It remains to choose the lower endpoint. When $\lambda \geq c_1 p_N^{-2/(d - 2)}$, then $p_N |S_\lambda| \gtrsim_d p_N \lambda^{d/2 - 1} \geq c_1^{(d - 2)/2}.$
By picking a sufficiently large $c_1 = c_1(d,q)$, we have
$p_N |S_\lambda| \geq 1$
throughout the stated range. Since $\gamma < d - 2$, the lower endpoint in the stated range is $o(N^2)$, so the choice of $c_2$ above is compatible with the range for all sufficiently large $N$.

We now prove the boundedness estimate. For $\b{x} \in Q_N$, set
\begin{align}
\mathcal N_{\lambda,N}^{(q)}(\b{x})
    & = \{\b{y} \in Q_N : \b{y} - \b{x} \in q S_\lambda\}, &
\Delta_{\lambda,N}^{(q)}(\b{x})
    & = |\mathcal N_{\lambda,N}^{(q)}(\b{x})|.
\end{align}
Then
\begin{align}
0
    & \leq \Delta_{\lambda,N}^{(q)}(\b{x})
    \leq |S_\lambda|, &
|E_{\lambda,N}^{(q)}|
    & = \sum_{\b{x} \in Q_N} \Delta_{\lambda,N}^{(q)}(\b{x}).
\end{align}
For $\b{y} \in Q_N$, let $X_{\b{y}} = \one_{Q_{N,\theta}}(\b{y})$. Then the variables $X_{\b{y}}$ are independent Bernoulli variables with mean $\theta$, and for each $\b{x} \in Q_N$,
\begin{align}
\deg_{G_{\lambda,N}^{(q)}}(\b{x}, Q_{N,\theta})
    & = \sum_{\b{y} \in \mathcal N_{\lambda,N}^{(q)}(\b{x})} X_{\b{y}}.
\end{align}
Therefore,
\begin{align}
\mathbb E \deg_{G_{\lambda,N}^{(q)}}(\b{x}, Q_{N,\theta})^2
    & = \theta \Delta_{\lambda,N}^{(q)}(\b{x}) + \theta^2 \Delta_{\lambda,N}^{(q)}(\b{x}) \bigl(\Delta_{\lambda,N}^{(q)}(\b{x}) - 1\bigr) \\
    & \leq \theta \Delta_{\lambda,N}^{(q)}(\b{x}) + \theta^2 \Delta_{\lambda,N}^{(q)}(\b{x})^2.
\end{align}
Summing over $\b{x} \in Q_N$ gives
\begin{align}
\mathbb E \sum_{\b{x} \in Q_N} \deg_{G_{\lambda,N}^{(q)}}(\b{x}, Q_{N,\theta})^2
    & \leq \theta |E_{\lambda,N}^{(q)}| + \theta^2 \sum_{\b{x} \in Q_N} \Delta_{\lambda,N}^{(q)}(\b{x})^2.
\end{align}
Since $\Delta_{\lambda,N}^{(q)}(\b x) \leq |S_\lambda|$ and $|E_{\lambda,N}^{(q)}| \simeq_{d,q} |Q_N| |S_\lambda|.$,
\begin{align}
\sum_{\b{x} \in Q_N} \Delta_{\lambda,N}^{(q)}(\b{x})^2
    & \leq |Q_N| |S_\lambda|^2
    \lesssim_{d,q} \frac{|E_{\lambda,N}^{(q)}|^2}{|Q_N|}.
\end{align}
For the diagonal term, $\theta \geq p_N$, as well as $p_N |S_\lambda| \geq 1$ and $|E_{\lambda,N}^{(q)}| \simeq_{d,q} |Q_N| |S_\lambda|,$ imply
\begin{align}
\theta |E_{\lambda,N}^{(q)}|
    & = \frac{|Q_N|}{\theta |E_{\lambda,N}^{(q)}|} \theta^2 \frac{|E_{\lambda,N}^{(q)}|^2}{|Q_N|} \\
    & \lesssim_{d,q} \frac{1}{\theta |S_\lambda|} \theta^2 \frac{|E_{\lambda,N}^{(q)}|^2}{|Q_N|} \\
    & \lesssim_{d,q} \theta^2 \frac{|E_{\lambda,N}^{(q)}|^2}{|Q_N|}.
\end{align}
Combining the second-moment decomposition above with the last two estimates completes the proof.
\end{proof}

\begin{rmk}\label{rmk:prob-adjust}
The normalization $p_N=N^{-\gamma}$ is inessential. For any fixed constant $a>0$, independent of $N$, the lemma and its proof remain valid with $p_N=aN^{-\gamma}$, provided that $N$ is then taken sufficiently large.
\end{rmk}

\begin{crlr}\label{cor:graph-bddness-cnst}
Let $d\geq 5$, let $0<\gamma<d-2$, put $p_N=N^{-\gamma}$, and fix $q\in\N$ and $a>0$. There are constants $c_0=c_0(d,q)>0$, $c_1=c_1(d,q,a)>0$, and $c_2=c_2(d,q)>0$ such that, uniformly for all integers $\lambda$ satisfying
\begin{align}
c_1p_N^{-2/(d-2)}
    & \leq \lambda
    \leq c_2N^2,
\end{align}
and all $ap_N\leq \theta \leq 1$, we have
\begin{align}
\mathbb E\sum_{\b x \in Q_N}\deg_{G_{\lambda,N}^{(q)}}(\b x ,Q_{N,\theta})^2
    \leq c_0\theta^2\frac{|E_{\lambda,N}^{(q)}|^2}{|Q_N|}.
\end{align}
\end{crlr}

\begin{proof}
By Lemma \ref{lm:graph-bddness} and Remark \ref{rmk:prob-adjust} for $\widetilde p_N=ap_N$, the lower endpoint in the range of $\lambda$ becomes $ \widetilde p_N^{-2/(d-2)}=a^{-2/(d-2)}p_N^{-2/(d-2)}, $ and the factor $a^{-2/(d-2)}$ may be absorbed into $c_1$. The remaining conclusions are unchanged.
\end{proof}
\smallskip 

The preceding estimates provide the deterministic density and second-moment inputs required for sparse transference. Since the argument must be applied simultaneously to a family of graphs indexed by the admissible values of $\lambda$, we record a quantitative graph form of Schacht's theorem with uniform constants and an exponentially small failure probability. Its conclusion also preserves a positive proportion of the expected edge count, as required for the subsequent counting estimate.

\begin{lemma}\label{lm:quant-schacht-graphs}
Let $0<\alpha<1$, let $0<\delta<1-\alpha$, let $0<\zeta\leq1$, and let $K\geq1$. There exist constants $C_{\mathrm{tr}}\geq1$, $\xi>0$, and $b>0$, depending only on $\alpha$, $\delta$, $\zeta$, and $K$, with the following property. For each $N$, let $V_N$ be a finite set with $v_N=|V_N|\to\infty$, $\pi_N\in(0,1)$, and $\mathcal H_N$ be a nonempty finite family of simple graphs on $V_N$. For $H\in\mathcal H_N$ and $U\subseteq V_N$, write $e_H(U)=|E(H[U])|$. Assume that, for all sufficiently large $N$, the following conditions hold uniformly for $H\in\mathcal H_N$:
\begin{enumerate}[label=\textnormal{(\roman*)}]
\item  \label{pt:quant-schacht-density}
    every $U\subseteq V_N$ with $|U|\geq\alpha v_N$ satisfies
    $e_H(U)\geq\zeta|E(H)|$
\item \label{pt:quant-schacht-bddness}
    for every $\pi_N\leq\theta\leq1$, if $V_{N,\theta}$ is the binomial random subset of $V_N$ with retention probability $\theta$, then
    \begin{align}
    \mathbb E\sum_{v\in V_N}\deg_H(v,V_{N,\theta})^2
    \leq K\theta^2\frac{|E(H)|^2}{v_N};
    \end{align}
\item \label{pt:quant-schacht-divergence}
    we have $\displaystyle \min_{H\in\mathcal H_N}\pi_N^2|E(H)|\to\infty.$
\end{enumerate}
Let $\omega_N\to\infty$, and $p_N$ satisfy $ C_{\mathrm{tr}}\pi_N\leq p_N\leq1/\omega_N.$
If $\log|\mathcal H_N|=o(p_Nv_N),$
then, for all sufficiently large $N$, with probability at least $1-|\mathcal H_N|2^{-bp_Nv_N}$, the binomial random subset $\Omega_N=V_{N,p_N}$ has the following simultaneous property: for every $H\in\mathcal H_N$ and every $W\subseteq\Omega_N$ such that $|W|\geq(\alpha+\delta)|\Omega_N|,$
we have $e_H(W)\geq\xi p_N^2|E(H)|.$
\end{lemma}

\begin{proof}
This is the case $k=i=2$ of \cite{Schacht16}*{Lemma 3.4}, followed by a union bound. We give the specialization. Enumerate the pairs $(N,H)$ with $H\in\mathcal H_N$ so that the corresponding values of $N$ are nondecreasing, and assign to the term corresponding to $(N,H)$ the baseline probability $\pi_N$, the sampling probability $p_N$, and the auxiliary parameter $\omega_N$. Since each $\mathcal H_N$ is finite, the corresponding values of $N$ tend to infinity along this enumeration. By part \ref{pt:quant-schacht-density}, the resulting graph sequence is $\alpha$-dense in Schacht's sense, with the fixed density witness $\zeta$. In the graph case, the only boundedness condition is the condition with index $i=1$, and it is exactly part \ref{pt:quant-schacht-bddness}. The hypothesis $\pi_N^2|E(H)|\to\infty$ required there follows uniformly from part \ref{pt:quant-schacht-divergence}.

Apply \cite{Schacht16}*{Lemma 3.4} with $k=i=2$, $\beta_{\mathrm S}=1$, and $\gamma_{\mathrm S}=\alpha+\delta$. Then $\beta_{\mathrm S}\gamma_{\mathrm S}=\alpha+\delta$. 
We record the dependence of the constants in this application. For $k = 2$, the proof of \cite{Schacht16}*{Lemma 3.4} uses only the induction step from $i = 1$ to $i = 2$. The induction hypothesis is applied with density gap $\delta / 8$, so its base case requires a density witness at $\alpha + \delta / 64$; part \ref{pt:quant-schacht-density} supplies the same witness $\zeta$. Thus, in Schacht's notation, we may take $\xi' = \delta \zeta / 128$, $b' = (\delta / 8)^3 / 193$, and $C' = 1$. His Proposition 3.6, applied with $\eta = \delta^2 / 16$, gives $\hat b = \delta^2 / 64$. The choices in \cite{Schacht16}*{(14)--(19)} therefore give $C_{\mathrm{tr}}$, $\xi$, and $b$ depending only on $\delta$, $\zeta$, and $K$, which suffices for the dependence stated here. The graph and probability sequences affect only the threshold beyond which the conclusion holds, not these three constants.
For each $H\in\mathcal H_N$, it follows that, with probability at least $1-2^{-bp_Nv_N}$, every $W\subseteq V_{N,p_N}$ with $|W|\geq(\alpha+\delta)|V_{N,p_N}|$ satisfies $e_H(W)\geq\xi p_N^2|E(H)|.$ For $H\in\mathcal H_N$ and $W\subseteq V_N$, Schacht's notation gives
\begin{align}
E_{V_N}^{2}(W)
    & = \left\{ e\in E(H[V_N]): |e\cap W|\geq2 \right\} \\
    & = \left\{ e\in E(H): e\subseteq W \right\} 
    = E(H[W]).
\end{align}
Indeed, $H[V_N]=H$, and, since every edge of $H$ has exactly two vertices, the condition $|e\cap W|\geq2$ is equivalent to $e\subseteq W$. Thus $\left|E_{V_N}^{2}(W)\right| = e_H(W).$

For each $H\in\mathcal H_N$, let $\mathcal B_{N,H}$ be the event that there exists a set $W\subseteq\Omega_N$ satisfying
$|W|\geq(\alpha+\delta)|\Omega_N|$
but
$ e_H(W)<\xi p_N^2|E(H)|.$
Applying \cite{Schacht16}*{Lemma 3.4} with $k=i=2$, $\beta_{\mathrm S}=1$, $\gamma_{\mathrm S}=\alpha+\delta$, and $U=V_N$ gives $\mathbb P(\mathcal B_{N,H}) \leq 2^{-bp_Nv_N}$ for every $H\in\mathcal H_N$, with the same constants $\xi$ and $b$ for all $H$. Therefore, by the union bound,
\begin{align}
\mathbb P\left( \bigcup_{H\in\mathcal H_N}\mathcal B_{N,H} \right)
    &\leq \sum_{H\in\mathcal H_N} \mathbb P(\mathcal B_{N,H}) \\
    &\leq |\mathcal H_N|2^{-bp_Nv_N}.
\end{align}

Since every $H\in\mathcal H_N$ is a simple graph on $v_N$ vertices, $|E(H)| \leq v_N^2/2.$
Hence
\begin{align}
\min_{H\in\mathcal H_N} \pi_N^2|E(H)|
    \leq \frac{(\pi_Nv_N)^2}{2}.
\end{align}
The divergence hypothesis therefore implies $\pi_Nv_N\to\infty$. Since $p_N\geq C_{\mathrm{tr}}\pi_N$, it follows that $p_Nv_N\to\infty$. Moreover,
\begin{align}
|\mathcal H_N|2^{-bp_Nv_N}
    = \exp\left( \left( \frac{\log|\mathcal H_N|}{p_Nv_N}
    - b\log2 \right) p_Nv_N \right).
\end{align}
By the hypothesis $\log|\mathcal H_N|=o(p_Nv_N)$, the last expression tends to zero. Thus the simultaneous conclusion holds with probability at least $1-|\mathcal H_N|2^{-bp_Nv_N}$.
\end{proof}

We now combine the dense supersaturation estimate with the graph boundedness bounds and the quantitative form of Schacht's transference theorem. The resulting proposition transfers the multidilate spherical count to every sufficiently large subset of the random set, uniformly over all admissible values of $\lambda$. This is the principal sparse input needed for the proof of the main theorem.

\begin{prop}\label{prop:sparse-rel}
Let $d\geq 5$, let $0<\gamma<d-2$, and put $p_N=N^{-\gamma}$. Let $\Omega_N\subseteq Q_N$ be the binomial random subset obtained by retaining each point independently with probability $p_N$. For every $0<\rho<1$, there exist integers $s>r\geq2$ and $J\geq0$ depending only on $d$ and $\rho$, constants $c_0=c_0(d,\rho)>0$, $c_1=c_1(d,\rho)>0$, $c_2=c_2(d,\rho)>0$, and $\mathfrak c =\mathfrak c (d,\rho,\gamma)>0$ so that, for all sufficiently large $N$, with probability at least $1-\exp(-\mathfrak c N^{d-\gamma})$, every subset $B\subseteq\Omega_N$ with $|B|\geq\rho p_N|Q_N|$
satisfies
\begin{align}
\frac{1}{|Q_N|} \sum_{\b x \in\Z^d} \one_B(\b x )
    \sum_{j=0}^{J} A_{s^{2(J-j)}\lambda}^{(r^j)} \one_B(\b x )
\geq c_0p_N^2,
\end{align}
for every integer $\lambda$ satisfying $c_1p_N^{-2/(d-2)} \leq \lambda \leq c_2N^2$. Thus, $D^2(B)\cap \big\{ r^{2j}s^{2(J-j)}\lambda \, : \, 0\leq j\leq J \big\} \neq \varnothing$ for every integer $\lambda$ in the allowed range.
\end{prop}

\begin{proof}
Apply Proposition \ref{prop:dense-lm-supersat} with density parameter $\rho/4$. We obtain integers $s>r\geq2$ and $J\geq0$, and constants $\zeta>0$, $c_{\mathrm d}>0$, and $\Lambda_0\in\N$, depending only on $d$ and  $\rho$, such that every $U\subseteq Q_N$ with $|U|\geq\rho|Q_N|/4$ satisfies
\begin{align}
    \frac{1}{|Q_N|}\sum_{\b x \in\Z^d}\one_U(\b x )\sum_{j=0}^{J}A_{s^{2(J-j)}\lambda}^{(r^j)}\one_U(\b x )\geq\zeta
    \label{eq:sparse-rel-dense-input}
\end{align}
whenever $N$ is sufficiently large and $\Lambda_0\leq\lambda\leq c_{\mathrm d}N^2$.

For $0\leq j\leq J$, define $u_j:=s^{2(J-j)}$, and 
\begin{align}
\mathcal E_{j,\lambda,N}
    & = \left\{(\b x ,\b y )\in Q_N\times Q_N:
        \b y -\b x \in r^jS_{u_j\lambda}\right\}, \\
e_{j,\lambda,N}(U)
    & = \left|\mathcal E_{j,\lambda,N}\cap(U\times U)\right|, \\
M_{\lambda,N}
    & = \sum_{j=0}^{J}|\mathcal E_{j,\lambda,N}|, \\
e_{\lambda,N}(U)
    & = \sum_{j=0}^{J}e_{j,\lambda,N}(U).
\end{align}
By the definition of the spherical averages,
\begin{align}
\frac{1}{|Q_N|} \sum_{\b x \in\Z^d} \one_U(\b x ) 
    \sum_{j=0}^{J} A_{u_j\lambda}^{(r^j)}\one_U(\b x )
= \frac{1}{|Q_N|}\sum_{j=0}^{J}\frac{e_{j,\lambda,N}(U)}{|S_{u_j\lambda}|}.
\label{eq:sparse-rel-counting-identity}
\end{align}

Let $H_{\lambda,N}$ be the simple graph on $Q_N$ in which $\{\b x ,\b y \}$ is an edge when $\b y -\b x \in r^jS_{u_j\lambda}$ for some $0\leq j\leq J$. The squared distance associated with the $j$th relation is
$r^{2j}u_j\lambda = r^{2j}s^{2(J-j)}\lambda.$
These values are pairwise distinct because $s>r$. Hence the constituent edge relations are disjoint, and their symmetry gives
\begin{align}
M_{\lambda,N} & = 2|E(H_{\lambda,N})|,&
e_{\lambda,N}(U) & = 2|E(H_{\lambda,N}[U])|.
\label{eq:sparse-rel-directed-undirected}
\end{align}

For each $0\leq j\leq J$, let $K_j$ be the boundedness constant supplied by Corollary \ref{cor:graph-bddness-cnst} for the dilation $r^j$; this constant is independent of the fixed constant factor in that corollary. Since the numbers $r^j$ and $u_j$ range over fixed finite sets, Lemma \ref{lm:graph-bddness} also supplies constants $\kappa_1,\kappa_2>0$, depending only on $d$, $r$, $s$, and $J$, such that, whenever its size estimates hold for every $j$,
\begin{align}
\kappa_1|Q_N|\max_{0\leq j\leq J}|S_{u_j\lambda}|
    \leq M_{\lambda,N}
    \leq \kappa_2|Q_N|\min_{0\leq j\leq J}|S_{u_j\lambda}|.
\label{eq:sparse-rel-M-comparison}
\end{align}
Set
\begin{align}
\zeta_{\mathrm{gr}}
    = \min\left\{1,\frac{\zeta}{\kappa_2}\right\}, &
K_{\mathrm{gr}}
    = \max\left\{1,4(J+1)\sum_{j=0}^{J}K_j\right\}.
\end{align}
Apply Lemma \ref{lm:quant-schacht-graphs} with $\alpha = \delta = \rho/4$, $\zeta=\zeta_{\mathrm{gr}}$, and $K=K_{\mathrm{gr}}$,
and let $C_{\mathrm{tr}}$, $\xi$, and $b$ be the resulting constants. 

We now choose the constants in the range of $\lambda$. Apply Lemma \ref{lm:graph-bddness} to each dilation $r^j$ and Corollary \ref{cor:graph-bddness-cnst} to each $r^j$ with the fixed factor $C_{\mathrm{tr}}^{-1}$.
Since the index set $\{j:0\leq j\leq J\}$ is finite, we may choose $c_1=c_1(d,\rho)>0$ and $c_2=c_2(d,\rho)>0$ such that, for all sufficiently large $N$, whenever
\begin{align}
c_1p_N^{-2/(d-2)}
    & \leq \lambda
    \leq c_2N^2,
\label{eq:sparse-rel-common-range}
\end{align}
we have $\Lambda_0\leq\lambda\leq c_{\mathrm d}N^2$, the comparison \eqref{eq:sparse-rel-M-comparison} holds, and, for every $0\leq j\leq J$ and every ${p_N}/{C_{\mathrm{tr}}}\leq\theta\leq1$,
\begin{align}
\mathbb E\sum_{\b x \in Q_N}
    \deg_{G_{u_j\lambda,N}^{(r^j)}}(\b x ,Q_{N,\theta})^2
\leq K_j\theta^2\frac{|\mathcal E_{j,\lambda,N}|^2}{|Q_N|}.
\label{eq:sparse-rel-individual-bddness}
\end{align}
The constants can also be chosen so that, uniformly in this range,
\begin{align}
|S_{u_j\lambda}| \simeq_{d,r,s,J} \lambda^{d/2-1}
\label{eq:sparse-rel-sphere-comparison}
\end{align}
for every $0\leq j\leq J$. The interval in \eqref{eq:sparse-rel-common-range} is nonempty for all sufficiently large $N$, since
\begin{align}
p_N^{-2/(d-2)}
    & = N^{2\gamma/(d-2)}
    = o(N^2).
\end{align}

Discarding finitely many values of $N$, define
\begin{align}
\mathcal H_N
    = \left\{H_{\lambda,N}:\lambda\in\N
        \text{ satisfies \eqref{eq:sparse-rel-common-range}}\right\}.
\end{align}
We verify the hypotheses of Lemma \ref{lm:quant-schacht-graphs} for this family. First, let $U\subseteq Q_N$ satisfy $|U|\geq\rho|Q_N|/4$. By \eqref{eq:sparse-rel-dense-input}, \eqref{eq:sparse-rel-counting-identity}, and \eqref{eq:sparse-rel-M-comparison},
\begin{align}
e_{\lambda,N}(U)
    & \geq \zeta|Q_N| \min_{0\leq j\leq J} |S_{u_j\lambda}| \\
    & \geq \frac{\zeta}{\kappa_2} M_{\lambda,N}.
\end{align}
Using \eqref{eq:sparse-rel-directed-undirected}, we obtain $|E(H_{\lambda,N}[U])| \geq \zeta_{\mathrm{gr}}|E(H_{\lambda,N})|.$

Second, for every $W\subseteq Q_N$ we have
\begin{align}
\deg_{H_{\lambda,N}}(\b x ,W)
    = \sum_{j=0}^{J}\deg_{G_{u_j\lambda,N}^{(r^j)}}(\b x ,W).
\end{align}
Consequently, for every ${p_N}/{C_{\mathrm{tr}}} \leq\theta\leq1$, Cauchy's inequality and \eqref{eq:sparse-rel-individual-bddness} give
\begin{align}
\mathbb E\sum_{\b x \in Q_N}\deg_{H_{\lambda,N}}(\b x ,Q_{N,\theta})^2
    & \leq (J+1) \sum_{j=0}^{J} K_j \theta^2  
        \frac{|\mathcal E_{j,\lambda,N}|^2}{|Q_N|} \\
    & \leq K_{\mathrm{gr}} \theta^2 \frac{|E(H_{\lambda,N})|^2}{|Q_N|}.
\end{align}

Third, \eqref{eq:sparse-rel-M-comparison}, \eqref{eq:sparse-rel-sphere-comparison}, and the lower bound in \eqref{eq:sparse-rel-common-range} imply, uniformly in admissible $\lambda$,
\begin{align}
\left(\frac{p_N}{C_{\mathrm{tr}}}\right)^2|E(H_{\lambda,N})|
    & \gtrsim_{d,\rho} p_N^2|Q_N|\lambda^{(d-2)/2} \\
    & = \bigl(p_N|Q_N|\bigr)\bigl(p_N\lambda^{(d-2)/2}\bigr) \\
    & \gtrsim_{d,\rho} N^{d-\gamma}
    \longrightarrow \infty.
\end{align}
Finally, the family contains at most $c_2N^2+1$ graphs, and hence
$\log|\mathcal H_N| = O(\log N) = o(p_N|Q_N|).$
Taking $\omega_N=p_N^{-1/2}$, we have
\begin{align}
C_{\mathrm{tr}}\frac{p_N}{C_{\mathrm{tr}}}
    = p_N
    \leq p_N^{1/2}
    = \frac{1}{\omega_N}
\end{align}
for all sufficiently large $N$. Lemma \ref{lm:quant-schacht-graphs} therefore gives, with probability at least $1-(c_2N^2+1)2^{-bp_N|Q_N|}$, simultaneously for every admissible $\lambda$, the implication
\begin{align}
W\subseteq\Omega_N, \quad 
|W|\geq\frac{\rho}{2}|\Omega_N| \quad 
\Longrightarrow \quad  
|E(H_{\lambda,N}[W])|
    \geq \xi p_N^2|E(H_{\lambda,N})|.
\label{eq:sparse-rel-transferred-edges}
\end{align}

Chernoff's inequality gives
\begin{align}
\mathbb P\left(|\Omega_N|>2p_N|Q_N|\right)
    \leq \exp\left(-\frac{1}{3}p_N|Q_N|\right),
\end{align}
because $p_N|Q_N|\simeq_d N^{d-\gamma}\to\infty$. Since $\log N=o(p_N|Q_N|)$, there exists $\mathfrak c =\mathfrak c (d,\rho,\gamma)>0$ such that the intersection of the event in \eqref{eq:sparse-rel-transferred-edges} and $\{|\Omega_N|\leq 2p_N|Q_N|\} $ has probability at least $1-\exp(-\mathfrak c N^{d-\gamma})$ for all sufficiently large $N$. On this intersection, let $B\subseteq\Omega_N$ satisfy $|B|\geq\rho p_N|Q_N|$. Then $|B|\geq\frac{\rho}{2}|\Omega_N|,$
so \eqref{eq:sparse-rel-transferred-edges} applies to $B$ for every admissible $\lambda$. Using \eqref{eq:sparse-rel-counting-identity}, \eqref{eq:sparse-rel-directed-undirected}, \eqref{eq:sparse-rel-transferred-edges}, and the lower estimate in \eqref{eq:sparse-rel-M-comparison}, we obtain
\begin{align}
\frac{1}{|Q_N|} \sum_{\b x \in\Z^d} \one_B(\b x ) 
        \sum_{j=0}^{J} A_{u_j\lambda}^{(r^j)} \one_B(\b x )
    & \geq \frac{e_{\lambda,N}(B)}{|Q_N|
        \max_{0\leq j\leq J}|S_{u_j\lambda}|} \\
    & = \frac{2|E(H_{\lambda,N}[B])|}{|Q_N|
        \max_{0\leq j\leq J}|S_{u_j\lambda}|} \\
    & \geq \xi p_N^2\frac{M_{\lambda,N}}{|Q_N|
        \max_{0\leq j\leq J}|S_{u_j\lambda}|} \\
    & \geq \xi\kappa_1p_N^2.
\end{align}
Thus the desired estimate holds with $c_0=\xi\kappa_1$.

Finally, the left-hand side of the preceding estimate is positive and all its summands are nonnegative. Hence there exist $0\leq j\leq J$, $\b x \in B$, and $\b{v}\in S_{s^{2(J-j)}\lambda}$ such that $\b y =\b x +r^j\b{v}\in B$. Thus
$|\b x -\b y |^2 = r^{2j}|\b{v}|^2 = r^{2j}s^{2(J-j)}\lambda,$
which proves the distance conclusion. The dependencies of $c_0$, $c_1$, $c_2$, and $\mathfrak c $ are exactly those stated in the proposition.
\end{proof}

\section{Proof of the Main Results}
Using the ingredients provided in the previous sections, we are now ready to state the proofs of our main results.

\begin{proof}[Proof of Theorem \ref{mainthm:multi-count}]
Apply Proposition \ref{prop:sparse-rel} with density parameter $\rho=\delta/2$. We obtain integers $s>r\geq2$ and $J\geq0$, and constants $c_0>0$, $c_1>0$, $c_2>0$, and $\widetilde c >0$, with the dependencies stated in the theorem, such that, with probability at least $1-\exp(-\widetilde c N^{d-\gamma})$, every set $B\subseteq\Omega_N$ satisfying $ |B|\geq\rho p_N|Q_N|$ obeys \eqref{eq:thm-multi-count} for every integer $\lambda$ in the stated range.

Since $|Q_N|\simeq_d N^d$ and $p_N=N^{-\gamma}$, we have $ p_N|Q_N|\simeq_d N^{d-\gamma}\longrightarrow\infty.$ Therefore Chernoff's inequality gives
\begin{align}
\mathbb P\left( |\Omega_N|<\frac12p_N|Q_N| \right)
    \leq \exp\left(-\frac18p_N|Q_N|\right)
    \leq \exp\left(-c_{\mathrm{ch}}N^{d-\gamma}\right)
\end{align}
for some constant $c_{\mathrm{ch}}=c_{\mathrm{ch}}(d)>0$. Hence, after setting $\mathfrak c =\frac12\min\{\widetilde c ,c_{\mathrm{ch}}\}$ and taking $N$ sufficiently large, with probability at least $1-\exp(-\mathfrak c N^{d-\gamma})$, both the conclusion of Proposition \ref{prop:sparse-rel} and the estimate $ |\Omega_N|\geq p_N|Q_N| /2$ hold simultaneously.

On this event, let $B\subseteq\Omega_N$ satisfy $|B|\geq\delta|\Omega_N|$. Then
\begin{align}
|B| \geq \delta|\Omega_N|
    \geq \frac{\delta}{2}p_N|Q_N|
    = \rho p_N|Q_N|.
\end{align}
Proposition \ref{prop:sparse-rel} therefore applies to $B$ and yields the desired inequality simultaneously for every integer $\lambda$ in the stated range.

It remains to establish the final optimality assertion. First, the admissible interval for $\lambda$ contains an integer for all sufficiently large $N$. Indeed, since $p_N = N^{-\gamma}$ and $\gamma < d - 2$, 
\begin{align}
\frac{c_1 p_N^{-2/(d-2)}}{c_2 N^2}
    & = \frac{c_1}{c_2} N^{-2(d-2-\gamma)/(d-2)}
    \longrightarrow 0.
\end{align}
Since the upper endpoint tends to infinity, the length of the admissible interval therefore tends to infinity. Fix an admissible integer $\lambda_N$ for each sufficiently large $N$.

Suppose that $\Omega_N = \varnothing$, and take $B = \varnothing$. Then $|B| = 0 = \delta |\Omega_N|$, so $B$ satisfies the required relative-density condition. However,
\begin{align}
\frac{1}{|Q_N|} \sum_{\b{x} \in \Z^d} 
\one_B(\b{x}) \sum_{j=0}^{J} A_{s^{2(J-j)}
\lambda_N}^{(r^j)} \one_B(\b{x})
    & = 0  < c_0 p_N^2.
\end{align}
Thus the conclusion of the theorem fails whenever $\Omega_N = \varnothing$, and hence
\begin{align}
\rho_N
    & \geq \mathbb P\left(\Omega_N = \varnothing\right) 
      = (1 - p_N)^{|Q_N|}.
\end{align}
In particular, $\rho_N > 0$. Since $p_N \longrightarrow 0$ and $|Q_N| = (2N + 1)^d$, we have
\begin{align}
-\log \mathbb P\left(\Omega_N = \varnothing\right)
    &= -|Q_N|\log(1 - p_N) \\
    &= |Q_N|\left(p_N + O(p_N^2)\right) \\
    &= (1 + o(1))p_N |Q_N| \\
    &= (2^d + o(1))N^{d-\gamma}.
\end{align}
Consequently, $-\log \rho_N \leq (2^d + o(1))N^{d-\gamma}.$

On the other hand, the already established  probability estimate  
$\rho_N \leq \exp\left(-\mathfrak c N^{d-\gamma}\right), $
means that $-\log \rho_N \geq \mathfrak c N^{d-\gamma}.$
Combining the preceding two estimates, we obtain
\begin{align}
\mathfrak c N^{d-\gamma}
    &\leq -\log \rho_N \leq (2^d + o(1))N^{d-\gamma}.
\end{align}
Note that the upper bound on $\rho_N$ shows that $\rho_N < 1$, and so $|\log \rho_N| = -\log \rho_N  \simeq_{d,\delta,\gamma} N^{d-\gamma},$
which proves that the failure-probability estimate is optimal at the exponential scale.
\end{proof}
\bigskip 

\begin{proof}[Proof of Corollary \ref{maincrlr:multi-distances}]
Since the right-hand side of \eqref{eq:thm-multi-count} is positive and every summand on the left-hand side is nonnegative, there exist an index $0\leq j\leq J$, a point $\b x \in B$, and a vector $\b{v}\in S_{s^{2(J-j)}\lambda}$ such that $\b y=\b x+r^j\b{v}\in B.$ Consequently,
\begin{align}
|\b x-\b y|^2
    = r^{2j}|\b{v}|^2
    = r^{2j}s^{2(J-j)}\lambda
    = \left(r^js^{J-j}\right)^2\lambda.
\end{align}
Thus $D^2(B)\cap \left\{ r^{2j}s^{2(J-j)}\lambda: 0\leq j\leq J \right\} \neq\varnothing,$ which completes the proof.
\end{proof}
\bigskip 

\begin{proof}[Proof of Proposition \ref{prop:lower-sharp}]
Set $a_N = p_N \lambda_N^{(d - 2)/2}$, $\eta_N = a_N^{1/2}$, and $\varepsilon_N = 2\eta_N$. The assumption on $\lambda_N$ gives
$a_N\to 0.$
Let $\mathcal A_N$ denote the event that there exists a set $B_N \subseteq \Omega_N$ satisfying
$|B_N| \geq (1 - \varepsilon_N)|\Omega_N|$ and
$D^2(B_N) \cap \{ r^{2j}s^{2(J - j)}\lambda_N : 0 \leq j \leq J \} = \varnothing.$
We shall prove that there exist constants $C_{\mathrm{sh}} = C_{\mathrm{sh}}(d,r,s,J) > 0$ and $c_{\mathrm{sh}} > 0$ such that, for all sufficiently large $N$,
\begin{align}
\mathbb P(\mathcal A_N)
    &\geq
    1 - C_{\mathrm{sh}}a_N^{1/2}
    - \exp\left(-c_{\mathrm{sh}}N^{d - \gamma}\right).
\end{align}

For each $0 \leq j \leq J$, put $q_j = r^j s^{J - j}$. Let $H_N$ be the simple graph with vertex set $Q_N$ in which two distinct vertices $\b{x},\b{y} \in Q_N$ are adjacent whenever $|\b{x} - \b{y}|^2 = q_j^2\lambda_N$ for at least one $0 \leq j \leq J$.

The standard upper bound for the number of lattice points on a discrete sphere gives, uniformly for $0 \leq j \leq J$,
$|S_{q_j^2\lambda_N}| \lesssim_{d,r,s,J} \lambda_N^{d/2 - 1}.$
Hence, for every $\b{x} \in Q_N$,
\begin{align}
\deg_{H_N}(\b{x})
    &\leq \sum_{j = 0}^{J}|S_{q_j^2\lambda_N}| 
    \lesssim_{d,r,s,J} \lambda_N^{d/2 - 1}.
\end{align}
It follows that
\begin{align}
|E(H_N)|
    &= \frac{1}{2} \sum_{\b{x} \in Q_N} \deg_{H_N}(\b{x}) 
    \lesssim_{d,r,s,J} |Q_N|\lambda_N^{d/2 - 1}.
\end{align}

Let $X_N = |E(H_N[\Omega_N])|$ be the number of edges of $H_N$ whose two endpoints both belong to $\Omega_N$. Since the two endpoints of each edge are retained independently,
\begin{align}
\mathbb E X_N
    &= p_N^2|E(H_N)| \\
    &\lesssim_{d,r,s,J} p_N^2|Q_N|\lambda_N^{d/2 - 1} \\
    &= p_N\lambda_N^{(d - 2)/2}p_N|Q_N| \\
    &= a_Np_N|Q_N|.
\end{align}
Markov's inequality therefore gives
\begin{align}
\mathbb P\left(
X_N > \eta_Np_N|Q_N| \right)
    &\leq \frac{\mathbb E X_N} {\eta_Np_N|Q_N|} 
    \lesssim_{d,r,s,J} \frac{a_N}{\eta_N} 
    = \eta_N.
\end{align}

Moreover, $p_N|Q_N| = N^{-\gamma}(2N + 1)^d \to \infty.$
Chernoff's inequality gives
\begin{align}
\mathbb P\left( |\Omega_N| < \frac{1}{2}p_N|Q_N| \right)
    &\leq \exp\left(-\frac{1}{8}p_N|Q_N|\right)  
    \leq \exp\left(-c_{\mathrm{sh}}N^{d - \gamma}\right)
\end{align}
for some absolute constant $c_{\mathrm{sh}} > 0$.

Let $\mathcal G_N$ be the intersection of the events
$X_N \leq \eta_Np_N|Q_N|$ and
$|\Omega_N| \geq \frac{1}{2}p_N|Q_N|$.
Using the union bound, we get
$\mathbb P(\mathcal G_N) \geq 1 - C_{\mathrm{sh}}\eta_N - \exp\left(-c_{\mathrm{sh}}N^{d - \gamma}\right).$

Suppose that $\mathcal G_N$ occurs. Choose one endpoint from every edge of $H_N[\Omega_N]$, and let $R_N$ be the set of all selected endpoints. Then $R_N$ is a vertex cover of $H_N[\Omega_N]$ and
$|R_N| \leq X_N  \leq \eta_Np_N|Q_N|  \leq 2\eta_N|\Omega_N|.$
Define $B_N = \Omega_N \setminus R_N$. Since $R_N$ meets every edge of $H_N[\Omega_N]$, the set $B_N$ is independent in $H_N$. Furthermore, $|B_N| = |\Omega_N| - |R_N|  \geq (1 - 2\eta_N)|\Omega_N|  = (1 - \varepsilon_N)|\Omega_N|.$

Because $B_N$ is independent in $H_N$, no two points of $B_N$ have squared distance $q_j^2\lambda_N$ for any $0 \leq j \leq J$. Since
$q_j^2\lambda_N = r^{2j}s^{2(J - j)}\lambda_N,$
we obtain $D^2(B_N) \cap \{ r^{2j}s^{2(J - j)}\lambda_N : 0 \leq j \leq J \} = \varnothing.$
Thus $\mathcal G_N \subseteq \mathcal A_N$, which proves the claimed probability bound. 

This distance avoidance also makes the corresponding spherical count vanish. Fix $0 \leq j \leq J$. If $\b{x} \in B_N$ and $\b{v} \in S_{s^{2(J - j)}\lambda_N}$, then $|r^j\b{v}|^2 = r^{2j}|\b{v}|^2  = r^{2j}s^{2(J - j)}\lambda_N.$
Therefore $\b{x} + r^j\b{v}$ cannot also belong to $B_N$. Hence, for every $\b{x} \in \Z^d$ and every $0 \leq j \leq J$, $\one_{B_N}(\b{x}) A_{s^{2(J - j)}\lambda_N}^{(r^j)} \one_{B_N}(\b{x}) = 0.$
Summing over $\b{x}$ and $j$ gives
\begin{align}
\frac{1}{|Q_N|} \sum_{\b{x} \in \Z^d} \one_{B_N}(\b{x})
\sum_{j = 0}^{J} A_{s^{2(J - j)}\lambda_N}^{(r^j)}
\one_{B_N}(\b{x}) &= 0.
\end{align}

Since $a_N \to 0$, we have $\varepsilon_N \to 0$. Thus, for every fixed $0 < \delta < 1$, for all sufficiently large $N$, the event $\mathcal A_N$ implies the existence of a set $B_N \subseteq \Omega_N$ satisfying $|B_N| \geq \delta|\Omega_N|$ for which \eqref{eq:thm-multi-count} fails at $\lambda_N$. Hence, if $\mathcal F_N$ denotes this fixed-density failure event, then $\mathcal A_N \subseteq \mathcal F_N$, and therefore $\mathbb P(\mathcal F_N) \geq 1 - C_{\mathrm{sh}}a_N^{1/2} - \exp\left(-c_{\mathrm{sh}}N^{d - \gamma}\right).$

For this fixed-density failure event, the first-moment argument also gives the sharper estimate
\begin{align}
\mathbb P(\mathcal F_N) \geq 1 - C_{\mathrm{sh},\delta}a_N
    - \exp\left(-c_{\mathrm{sh}}N^{d - \gamma}\right)
\end{align}
for some $C_{\mathrm{sh},\delta} = C_{\mathrm{sh},\delta}(d,r,s,J,\delta) > 0$. Indeed, Markov's inequality gives
\begin{align}
\mathbb P\left( X_N > \frac{1 - \delta}{2}p_N|Q_N| \right)
    &\leq \frac{2\mathbb E X_N}
    {(1 - \delta)p_N|Q_N|} 
    \lesssim_{d,r,s,J,\delta}
    a_N.
\end{align}
If simultaneously $|\Omega_N| \geq p_N|Q_N|/2$, then
$X_N \leq (1 - \delta)p_N|Q_N| /2 \leq (1 - \delta) |\Omega_N|.$
Deleting one endpoint from every edge therefore produces an independent set $B_N \subseteq \Omega_N$ satisfying
$|B_N| \geq |\Omega_N| - X_N  \geq \delta|\Omega_N|.$
Combining this observation with the Chernoff bound above proves the claimed probability bound and completes the proof.
\end{proof}
\bigskip

\begin{proof}[Proof of Theorem \ref{mainthm:sparse-random-patterns}]
Let $\mathcal Q$ be the finite set supplied by Corollary \ref{maincrlr:multi-distances}, and put $m=|\mathcal Q|$. Let $F_0=F-\min F=\{f-\min F:f\in F\}$, and put $L=\max F_0$. By van der Waerden's theorem \cite{Waerden27}, there exists an integer $W=W(m,L+1)$ such that every coloring of an interval of $W$ consecutive integers with $m$ colors contains a monochromatic nontrivial arithmetic progression of length $L+1$.

Define the interval of admissible integers by
\begin{align}
I_N = \left\{ \lambda\in\N: c_1p_N^{-2/(d-2)}
    \leq \lambda
    \leq c_2N^2 \right\}.
\end{align}
Since $p_N=N^{-\gamma}$ and $\gamma<d-2$, we have
\begin{align}
 p_N^{-2/(d-2)}
    & = N^{2\gamma/(d-2)}
    = o(N^2).
\end{align}
Consequently, $|I_N|\longrightarrow\infty$ as $N\to\infty$. In particular, $|I_N|\geq W$ for all sufficiently large $N$.

Now work on the event supplied by Corollary \ref{maincrlr:multi-distances}, and fix a subset $B\subseteq\Omega_N$ satisfying $|B|\geq\delta|\Omega_N|$. For every $\lambda\in I_N$, the corollary gives $D^2(B)\cap\left\{q^2\lambda:q\in\mathcal Q\right\}\neq\varnothing.$ Choose, for each $\lambda\in I_N$, one element $q_\lambda \in\mathcal Q$ such that $q_\lambda ^2\lambda\in D^2(B).$ Thus $\lambda\mapsto q_\lambda $ defines an $m$-coloring of $I_N$. Since $|I_N|\geq W$, van der Waerden's theorem provides integers $\lambda_0\in\N$ and $h\geq1$, and an element $q\in\mathcal Q$, such that $\lambda_0, \lambda_0+h, \ldots, \lambda_0+Lh\in I_N$ and $q_{\lambda_0+ih}=q$ for every $0\leq i\leq L$. It follows from the definition of the coloring that
\begin{align}
q^2\lambda_0+q^2hF_0\subseteq D^2(B).
\end{align}
Set $t=q^2h\in\N$ and $a=q^2\lambda_0-t\min F\in\Z$. Since $F_0=F-\min F$, we obtain
\begin{align}
a+tF=q^2\lambda_0+q^2hF_0\subseteq D^2(B).
\end{align}

The event from Corollary \ref{maincrlr:multi-distances} has probability at least $1-\exp(-\mathfrak c N^{d-\gamma})$ and holds simultaneously for every admissible subset $B$, so the same conclusion holds simultaneously for every $B\subseteq\Omega_N$ satisfying $|B|\geq\delta|\Omega_N|$.
\end{proof}
\bigskip

\begin{proof}[Proof of Theorem \ref{mainthm:sparse-random-aps}]
Let $\mathcal Q$ be the finite set supplied by Corollary \ref{maincrlr:multi-distances}, put $m=|\mathcal Q|$, and set $\alpha=m^{-1}$. We first establish the quantitative consequence of Szemerédi's theorem that will be used below. By Szemerédi's theorem \cite{Szemeredi75}, when $k\geq3$, and trivially when $k=2$, there exists an integer $M=M(k,\alpha)\geq k$ such that every subset of $\{0,1,\ldots,M-1\}$ of cardinality at least $\alpha M/4$ contains a nontrivial arithmetic progression of length $k$. We apply the standard averaging argument of Varnavides \cite{Varnavides59} to deduce a quantitative lower bound for the number of such progressions.

Let $I\subseteq\Z$ be an interval of $L$ consecutive integers, and let $A\subseteq I$ satisfy $|A|\geq\alpha L$. By translating, it is enough to consider $I=\{1,\ldots,L\}$. Assume that $L\geq8M/\alpha$, and put $D=\left\lfloor\alpha L/4M \right\rfloor.$
Let $\mathcal P_L$ be the family of all arithmetic progressions of length $M$ of the form $P_{c,r} = \left\{ c+jr: 0\leq j\leq M-1 \right\},$ for $1\leq r\leq D,$ and $1\leq c\leq L-(M-1)r.$
Define the central interval
\begin{align}
I^\circ = \left\{ (M-1)D+1,\ldots,L-(M-1)D \right\}.
\end{align}
Since $2(M-1)D\leq\alpha L/2$, we have $|A\cap I^\circ| \geq \alpha L/2.$
Every element of $I^\circ$ belongs to exactly $MD$ members of $\mathcal P_L$: for every $1\leq r\leq D$ and every $0\leq j\leq M-1$, the choice $c=x-jr$ gives a member $P_{c,r}$ containing $x$. Consequently,
\begin{align}
\sum_{P\in\mathcal P_L}|A\cap P|
    &\geq |A\cap I^\circ|MD \\
    &\geq \frac{\alpha}{2}LMD \\
    &\geq \frac{\alpha M}{2}|\mathcal P_L|,
\end{align}
where the last inequality follows from $|\mathcal P_L|\leq LD$.

Let
\begin{align}
\mathcal G_L = \left\{ P\in\mathcal P_L: |A\cap P|\geq\frac{\alpha M}{4} \right\}.
\end{align}
Using the preceding incidence estimate and the trivial bound $|A\cap P|\leq M$, we obtain
\begin{align}
\frac{\alpha M}{2}|\mathcal P_L|
    & \leq \sum_{P\in\mathcal P_L}|A\cap P| \\
    & \leq M|\mathcal G_L| + \frac{\alpha M}{4} 
        \left( |\mathcal P_L|-|\mathcal G_L| \right) \\
    &\leq M|\mathcal G_L| + \frac{\alpha M}{4}|\mathcal P_L|.
\end{align}
It follows that $|\mathcal G_L| \geq \frac{\alpha}{4}|\mathcal P_L|.$
Moreover, since $(M-1)D\leq\alpha L/4\leq L/4$ and $D\geq\alpha L/(8M)$, we have
\begin{align}
|\mathcal P_L|
    &= \sum_{r=1}^{D} \left( L-(M-1)r \right) \\
    &\geq \frac{3}{4}LD \\
    &\geq \frac{3\alpha}{32M}L^2.
\end{align}
Combining these, we obtain
\begin{align}
|\mathcal G_L|
\geq
\frac{3\alpha^2}{128M}L^2.
\end{align}

For every $P_{c,r}\in\mathcal G_L$, the set $\left\{ j\in\{0,\ldots,M-1\}: c+jr\in A \right\}$ has cardinality at least $\alpha M/4$, and hence contains a nontrivial arithmetic progression of length $k$. Thus $A\cap P_{c,r}$ contains a nontrivial arithmetic progression of length $k$. Choose one such progression for each $P_{c,r}\in\mathcal G_L$.

A fixed nontrivial progression $a,a+b,\ldots,a+(k-1)b$ can be chosen from at most $M^2$ members of $\mathcal P_L$. Indeed, if it arises from $P_{c,r}$, then there exist integers $u$ and $v$ satisfying $0\leq u\leq M-1$ and $1\leq v\leq M-1$ such that $a=c+ur,$ and $b=vr.$ The pair $(u,v)$ determines $r=b/v$, whenever this is an integer, and then determines $c=a-ur$. Hence there are at most $M^2$ possible choices of $P_{c,r}$. It follows from the lower bound on $|\mathcal G_L|$ that $A$ contains at least
\begin{align}
\nu L^2, \qquad
\nu = \frac{3\alpha^2}{128M^3} >0,
\label{eq:varnavides}
\end{align}
distinct nontrivial arithmetic progressions of length $k$. By translation, the same conclusion holds for every interval $I\subseteq\Z$ of length $L$ and every subset $A\subseteq I$ satisfying $|A|\geq\alpha L$.

We now apply \eqref{eq:varnavides} to the admissible scales. Define
\begin{align}
I_N
    & = \left\{ \lambda\in\N: c_1p_N^{-2/(d-2)}
    \leq \lambda
    \leq c_2N^2 \right\}, &
L_N
    & = |I_N|.
\end{align}
Since $p_N=N^{-\gamma}$ and $\gamma<d-2$, we have
\begin{align}
p_N^{-2/(d-2)}
& =
N^{2\gamma/(d-2)}
=
o(N^2).
\end{align}
Consequently, there exists a constant $\kappa=\kappa(d,\delta)>0$ such that $L_N \geq \kappa N^2$
for all sufficiently large $N$.

Work on the event supplied by Corollary \ref{maincrlr:multi-distances}, and fix a subset $B\subseteq\Omega_N$ satisfying $|B|\geq\delta|\Omega_N|$. For each $q\in\mathcal Q$, define
\begin{align}
I_{N,q}(B) = \left\{ \lambda\in I_N: q^2\lambda\in D^2(B) \right\}.
\end{align}
Corollary \ref{maincrlr:multi-distances} gives
\begin{align}
I_N = \bigcup_{q\in\mathcal Q}I_{N,q}(B).
\end{align}
Pigeonholing, we find an element $q_B\in\mathcal Q$ such that
\begin{align}
|I_{N,q_B}(B)|
    \geq \frac{L_N}{|\mathcal Q|}
    = \alpha L_N.
\end{align}
Applying \eqref{eq:varnavides} to $I_{N,q_B}(B)\subseteq I_N$, we obtain at least $\nu L_N^2$ pairs $(u,h)\in\Z^2$ satisfying $u\geq0$, $h\geq1$, and $u+th\in I_{N,q_B}(B)$ for every $0\leq t\leq k-1$. For every such pair, $q_B^2u+tq_B^2h = q_B^2(u+th) \in D^2(B)$ for every $0\leq t\leq k-1$. The map $(u,h) \longmapsto \left( q_B^2u,q_B^2h \right)$ is injective. Since $L_N \geq \kappa N^2$, we have that
\begin{align}
& \Big| \big\{ (a,b)\in\Z^2: a\geq0,\; b\geq1 \text{ and } a+tb\in D^2(B) \text{ for every } 0\leq t\leq k-1 \big\} \Big| \\
& \hspace*{1cm} \geq \nu L_N^2 \\
& \hspace*{1cm} \geq \nu\kappa^2N^4.
\end{align}
Setting $c_k = \nu\kappa^2 >0$ proves the asserted quantitative bound. Since $\alpha=|\mathcal Q|^{-1}$, the integer $M$, the constant $\nu$, and the constant $\kappa$ depend only on $d$, $\delta$, and $k$, so $c_k=c_k(d,\delta)$ has the stated dependence.

The event from Corollary \ref{maincrlr:multi-distances} has probability at least $1-\exp(-\mathfrak c N^{d-\gamma})$ and holds simultaneously for every subset $B\subseteq\Omega_N$ satisfying $|B|\geq\delta|\Omega_N|$. No additional probabilistic exceptional event was introduced, so the same probability lower bound holds for the conclusion of the theorem.
\end{proof}

\begin{proof}[Proof of Corollary \ref{maincrlr:prevalent-dilation}]
Work on the event supplied by Corollary \ref{maincrlr:multi-distances}, and fix a subset $B \subseteq \Omega_N$ satisfying $|B| \geq \delta |\Omega_N|$. For each $q \in \mathcal Q$, define
\begin{align}
I_{N,q}(B)
    &= \left\{ \lambda \in I_N : q^2 \lambda \in D^2(B) \right\}.
\end{align}
Corollary \ref{maincrlr:multi-distances} gives
\begin{align}
I_N
    &= \bigcup_{q \in \mathcal Q} I_{N,q}(B).
\end{align}
Pigeonholing, there exists $q_B \in \mathcal Q$ such that
$|I_{N,q_B}(B)| \geq |I_N| / |\mathcal Q|,$
which proves the first assertion.

The map $\lambda \mapsto q_B^2 \lambda$ is injective, and therefore
$|D^2(B)| \geq |I_{N,q_B}(B)| \geq |I_N| / |\mathcal Q|.$
Since $p_N^{-2/(d - 2)} = N^{2 \gamma/(d - 2)} = o(N^2)$, we have $|I_N| \gtrsim_{d,\delta} N^2$ for all sufficiently large $N$. Hence $|D^2(B)| \gtrsim_{d,\delta} N^2.$
On the other hand, since $B \subseteq Q_N$, every squared distance determined by $B$ is an integer between $0$ and $4 d N^2$. Thus
$|D^2(B)| \leq 4 d N^2 + 1  \lesssim_d N^2.$
Combining the two estimates gives $|D^2(B)| \simeq_{d,\delta} N^2$.

The event from Corollary \ref{maincrlr:multi-distances} has probability at least $1 - \exp(-\mathfrak c N^{d-\gamma})$ and holds simultaneously for every subset $B \subseteq \Omega_N$ satisfying $|B| \geq \delta |\Omega_N|$, so the same probability lower bound holds for the conclusion of the corollary.
\end{proof}

\medskip
\textbf{AI Disclosure}: In preparing this manuscript, the authors used GPT-6 Astra to assist with language editing, syntax, and proofreading of author-written drafts, and to identify potentially relevant literature. The mathematical arguments and results are the authors' own. All AI-generated suggestions were reviewed by the authors, and all literature references identified with AI assistance were independently verified. The authors take full responsibility for the content and accuracy of the manuscript.

\end{document}